\documentclass[12pt,leqno,oneside]{amsart}

\usepackage{amsmath,amsthm,amssymb,amsfonts,ifpdf,xcomment}

\ifpdf
\usepackage[pdftex,pdfstartview=FitH,pdfborderstyle={/S/B/W 1},%
colorlinks=true, linkcolor=blue, urlcolor=blue, citecolor=blue,%
pagebackref=true]{hyperref}
\else
\usepackage[hypertex]{hyperref}
\fi

\usepackage[abbrev,msc-links,nobysame]{amsrefs}

\newtheorem {thm}   {Theorem}
\newtheorem* {thm*}   {Theorem}
\newtheorem* {prp*}   {Proposition}
\newtheorem {lem}      [thm]    {Lemma}

\newtheorem {prp}[thm]  {Proposition}

\newcounter{AbcT}

\theoremstyle{definition}
\newtheorem*{nota}{Notation}

\renewcommand{\a}{\alpha}
\renewcommand{\b}{\beta}
\renewcommand{\d}{\delta}

\newcommand{\e}{\varepsilon}

\renewcommand{\l}{\lambda}

\newcommand{\s}{\sigma}

\newcommand{\R}{{\bf R}}

\newcommand{\Q}{{\bf Q}}

\newcommand{\Z}{{\bf Z}}

\newcommand{\E}{{\bf E}}

\renewcommand{\P}{{\bf P}}

\newcommand {\cN} {{\mathcal N}}

\newcommand {\cX} {{\mathcal X}}
\newcommand {\cY} {{\mathcal Y}}

\newcommand{\wt}{\widetilde}

\DeclareMathOperator{\supp}{supp}

\newcommand{\chiN}{{\chi_N}}
\newcommand{\chiM}{{\chi_M}}

\title[Bernoulli convolutions]%
{Absolute continuity of Bernoulli convolutions for algebraic parameters}
\author{P\'eter P. Varj\'u}
\thanks{
I gratefully acknowledge the support
of the Royal Society.
}


\begin{document}

\begin{abstract}
We prove that Bernoulli convolutions $\mu_\l$ are absolutely continuous provided
the parameter $\l$ is an algebraic number sufficiently close to $1$ depending
on the Mahler measure of $\l$.
\end{abstract}

\maketitle

\section{Introduction}\label{sc:intro}

Let $\l,p\in(0,1)$ be real numbers and let $\xi_1,\xi_2,\ldots$ be a sequence of independent
random variables with $\P(\xi_n=1)=p$, $\P(\xi_n=-1)=1-p$.
We define the Bernoulli convolution $\mu_{\l,p}$ with parameter $\l$ and $p$ as the law of the random variable
$\sum_{n=0}^\infty \xi_n\l^n$.

This paper is concerned with the problem of whether $\mu_{\l,p}$ is absolutely continuous or
singular with respect to the Lebesgue measure for given parameters $\l$ and $p$.
It is well-known that $\mu_{\l,p}$, like all self-similar measures, is of pure type.
The main result of the paper is the following.

\begin{thm}\label{th:main}
For every $\e>0$ and $p\in(0,1)$,
there is $c>0$ such that the following holds.
Let $\l<1$ be an algebraic number and suppose that
\[
\l> 1- c\min(\log M_\l,(\log M_\l)^{-1-\e}).
\]
Then the Bernoulli convolution $\mu_{\l,p}$ is absolutely continuous
with density in $L\log L$.
\end{thm}

Recall that a function $f:\R\to\R$ is in $L\log L$ if
\[
\int |f|\log (|f|+2)dx<\infty.
\]
In this theorem and everywhere in the paper, $M_\l$ stands for the
Mahler measure of an algebraic number $\l$, which is defined as follows.
Let $a(x-z_1)\cdots(x-z_d)$ be the minimal polynomial of $\l$.
Then
\[
M_\l=a\prod_{j:|z_j|>1}|z_j|.
\]

We note that the constant $c$ in the theorem may be taken to be a continuous function of the parameters $p$ and $\e$.

There were only very special explicit examples of absolutely continuous Bernoulli convolutions
known prior to this paper. (See the next section.)
In particular, this theorem gives the first explicit examples of biased (i.e. with $p\neq1/2$)
absolutely continuous Bernoulli
convolutions. 

In the course of the proof of Theorem \ref{th:main}, we keep track of the values of the constants
in a certain special case, which allows us to obtain completely explicit examples.
At this point, we only remark that $\mu_{1-10^{-50},p}$ is absolutely continuous provided $1/4\le p\le3/4$
and defer further discussion about more examples to Section \ref{sc:explicit}.

\subsection{Background}\label{sc:history}

For thorough surveys on Bernoulli convolutions we refer to \cite{60y}
and \cite{Sol-survey}.
For a discussion of the most recent developments, see \cite{Var-ECM}.
We limit this discussion to the case of unbiased Bernoulli convolutions, i.e. we take $p=1/2$
and omit the index $p$ in our notation.

Bernoulli convolutions originate in a paper of Jessen and Wintner \cite{JW-Riemann}
and they have been studied by Erd\H os in \cites{erdos39,erdos}.
If $\l<1/2$, then  $\supp\mu_\l$ is easily seen to be a Cantor set, hence $\mu_\l$ is a singular measure.
If $\l=1/2$, then $\mu_\l$ is the normalized Lebesgue measure restricted to the interval $[-2,2]$.

It has been noticed by Erd\H os \cite{erdos39} that $\mu_\l$ may be singular even if $\l>1/2$.
In particular, he showed that $\mu_\l$ is singular whenever $\l^{-1}$ is a Pisot number --
a real number greater than 1, all of whose Galois conjugates have absolute value strictly
less than $1$.
Pisot numbers have the property that their powers are approximated by integers with exponentially
small error.
This was exploited by Erd\H os to show that the Fourier transform of $\mu_\l$ does
not vanish at infinity.

It is easily seen that $\mu_{\l}=\mu_{\l^k}*\nu$, where $\nu$ is a probability measure
(the law of the random variable $\sum_{n:k\nmid n}\xi_n\l^n$).
Hence $\mu_{2^{-1/k}}$ is a convolution of $\mu_{1/2}$ with another probability measure,
and it is absolutely continuous.
Further explicit examples of absolutely continuous Bernoulli convolutions were
given by Garsia \cite{Garsia-arithmetic}*{Theorem 1.8}, who showed that
$\mu_\l$ is absolutely continuous, whenever $\l^{-1}$ is a real algebraic integer with Mahler measure 2.

The typical behaviour is absolute continuity for parameters in $(1/2,1)$.
Indeed, Erd\H os \cite{erdos} showed that $\mu_\l$ is absolutely continuous for almost all $\l\in(c,1)$,
where $c<1$ is an absolute constant.
This was extended by Solomyak \cite{Sol-Bernoulli} to almost all $\l\in(1/2,1)$.

Beside absolute continuity, another interesting problem is to determine the dimension of $\mu_\l$.
It was proved by Feng and Hu \cite{feng-hu}*{Theorem 2.8} that self-similar measures, and hence
Bernoulli convolutions in particular, are exact dimensional, that is the limit
\[
\lim_{r\to 0}\frac{\log\mu_\l([x-r,x+r])}{\log r}
\]
exists and is constant for $\mu_\l$-almost every $x$.
We call this limit the dimension of $\mu_\l$ and denote it by $\dim \mu_\l$.

Very recently, Hochman \cite{hochman} made a breakthrough on this problem.
He proved that the set of exceptional parameters
\[
\{\l\in(1/2,1):\dim\mu_\l<1\}
\]
is of packing dimension $0$.
Recall that a set of packing dimension $0$ is also a set of Hausdorff dimension $0$.
See \cite{Fal-fractal-geometry}*{Chapter 3.5} for the definition and basic properties
of packing dimension.
Building on this result, Shmerkin \cite{shmerkin} proved that 
\[
\{\l\in(1/2,1):\mu_\l\text{ is singular}\}
\]
is of Hausdorff dimension $0$.

Our aim in this paper is to  obtain results about the absolute continuity of $\mu_\l$ for
specific values of $\l$, in particular when $\l$ is algebraic.
The work of Hochman \cite{hochman} yields a formula for the dimension of $\mu_\l$ when $\l$ is
an algebraic number.
Denote by $h_\l$ the entropy of the random walk on the semigroup generated by the
transformations $x\mapsto \l\cdot x+1$ and $x\mapsto \l\cdot x-1$.
More precisely, let
\[
h_\l=\lim\frac{1}{n}H\Big(\sum_{i=0}^{n-1}\xi_i\l^i\Big),
\]
where $H(\cdot)$ denotes the Shannon entropy of a discrete random variable.
With this notation Hochman's formula is
\[
\dim \mu_\l=\min(-h_\l/\log\l,1).
\]
(See \cite{BV-entropy}*{Section 3.4}, where the formula is derived in this form from Hochman's main result.)

The quantity $h_\l$ has been studied in the paper \cite{BV-entropy}.
It was proved there \cite{BV-entropy}*{Theorem 5} that there is an absolute constant $c_0>0$ such that
for any algebraic number, we have
\[
c_0\cdot \min(\log M_\l,1)\le h_\l\le \min(\log M_\l,1).
\]
The $\log$'s in this formula as well as those that appear in the definition of entropy are
base $2$.
Numerical calculations reported in that paper indicate that one can take $c_0=0.44$.
This result combined with Hochman's formula implies that
$\dim\mu_\l=1$ provided $\l$ is an algebraic number with $1>\l>\min(2,M_\l)^{-1/c_0}$.

\subsection{The strategy of the proof}\label{sc:strategy}

We aim to take a step further and show that $\mu_\l$ is absolutely continuous
provided $\l$ is an algebraic number that satisfies the conditions of Theorem \ref{th:main}.

Unfortunately, the methods of \cite{shmerkin} do not seem to apply for specific parameters,
hence we need a different method.
We follow a strategy similar to Hochman's \cite{hochman} relying on lower bounds for the
entropy of convolution of measures.

We will work with the following notion of entropy.
Let $X$ be a bounded random variable and let $r>0$ be a real number.
We define
\[
H(X;r):=\int_0^1 H(\lfloor X/r+t\rfloor)dt.
\] 
On the right hand side $H(\cdot)$ denotes the Shannon entropy of a discrete random variable.
In addition, we define the conditional entropy
\[
H(X;r_1|r_2):=H(X;r_1)-H(X;r_2).
\] 
We will study the basic properties of these quantities in Section \ref{sc:entropy}.
By abuse of notation, we write $H(\mu;r_1|r_2)=H(X;r_1|r_2)$ and similar expressions
if $\mu$ denotes the law of $X$.

These quantities differ from those used by Hochman in that they involve an averaging
over a random translation.
This averaging endows these quantities with some useful properties as we will see in Section
\ref{sc:entropy-at-scale}, which often come in handy.
The idea of this averaging procedure originates in Wang's paper \cite{Wan-quantitative}*{Section 4.1}.

We fix an algebraic number $\l$ until the end of the section.
For a set $I\subset[0,1]$, we write $\mu^I$
for the law of the random variable
\[
\sum_{n\in\Z_{\ge 0}:\l^n\in I}\xi_n\l^n.
\]

The starting point of the proof is the observation going back at least to
Garsia \cite{Garsia-arithmetic}*{Lemma 1.51}
that any two points in the support of $\mu^{(\l^\ell,1]}$ are at distance at least
$c_\l \ell^{-a} M_\l^{-\ell}$, where $a$ is the number of Galois conjugates of $\l$ on the unit circle.
Taking any $\a<M_\l^{-1}$, this implies that
\[
H(\mu^{(\l^\ell,1]};\a^\ell)= H(\mu^{(\l^\ell,1]})\ge h_\l \cdot \ell
\]
for $\ell$ sufficiently large.
Indeed, the choice of $\a$ guarantees that any two possible values attained by the random variable
\[
\a^{-\ell}\sum_{n=0}^{\ell-1}\xi_n\l^n
\] 
are at distance at least $1$.
Hence taking integer parts does not change its Shannon entropy.

We choose another suitable number $\b>0$ and note the trivial bound
\[
H(\mu^{(\l^\ell,1]};\b^\ell)\le\log\b^{-1}\cdot \ell+C, 
\]
where $C$ is a constant depending on the length of the interval on which $\mu_\l$
is supported.
We make sure that $\log\b^{-1}<h_\l$ and take the difference of 
these bounds. Writing $x=\l^\ell$, we obtain
$H(\mu^{(x,1]};x^{c_1}|x^{c_2})\ge c_3\cdot \log x$, where $c_1$, $c_2$ and $c_3$ are
constants depending on $\a$, $\b$ and $h_\l$.
After rescaling we obtain
\[
H(\mu^{(xy,y]};x^{c_1}y|x^{c_2}y)\ge c_3\cdot \log x.
\]

Then we aim to find a suitable disjoint collection of intervals of the form $I_j=(x_jy_j,x_j]\subset (0,1]$
such that the corresponding intervals $(x_j^{c_1}y_j,x_j^{c_2}y_j]$ ``overlap a lot''.
At this stage we observe that
\[
\mu_\l=\mu^{I_1}*\cdots*\mu^{I_n}*\mu^{(0,1]\backslash \bigcup I_j}
\]
and
invoke some results about the growth of entropy under
convolution, which we formulate now.

Recall that the $\log$ function that appears in the definition of entropy is base $2$.
With this normalization, $1$ is the supremum of the conditional entropy
between two scales $s$ and $2s$ over the set of all probability measures.
The number $1-H(\mu;s|2s)$ measures how uniform $\mu$ is at scale $s$.
Our first result quantifies the following statement:
If two measures are very uniform at a given scale (and also at nearby scales), then their convolution is
even more uniform.

\begin{thm}\label{th:high-entropy-convolution}
There is an absolute constant $C>0$ such that the following holds.
Let $\mu,\wt\mu$ be two compactly supported probability measures on $\R$ and let $0<\a<1/2$ and $r>0$ be real numbers.
Suppose that
\[
H(\mu;s|2s)\ge 1-\a\quad \text{and}\quad H(\wt\mu;s|2s)\ge 1-\a
\]
for all $s$ with $|\log r -\log s|<3\log \a^{-1}$.

Then
\[
H(\mu*\wt\mu;r|2r)\ge 1- C(\log \a^{-1})^3 \a^2.\quad\footnote{One may take $C=10^{8}$ here.}
\]
\end{thm}

Our second result complements the first one.
We consider a probability measure that is uniform only at a small fraction of scales
and estimate how much its entropy grows when we convolve it with another
measure that has at least some small amount of entropy.

We write $\cN_1(X)$ for the maximal cardinality of a
collection of points in a set $X\subset \R$ that are of distance at least $1$ from each other.

\begin{thm}\label{th:low-entropy-convolution}
For every $0<\a<1/2$, there is a number $c>0$ such that the following holds.
Let $\mu,\nu$ be two compactly supported probability measures on $\R$.
Let  $\s_2<\s_1<0$ and $0< \b\le1/2$ be real numbers.
Suppose that 
\[
\cN_1\{ \s\in[\s_2,\s_1]:H(\mu;2^\s|2^{\s+1})>1-\a\}<c\b(\s_1-\s_2).
\quad\footnote{One may take $c=1/(1000\log\a^{-1})$ here.}
\]
Suppose further that
\[
H(\nu; 2^{\s_2}|2^{\s_1})>\b(\s_1-\s_2).
\]

Then
\[
H(\mu*\nu; 2^{\s_2}|2^{\s_1})> H(\mu;2^{\s_2}|2^{\s_1})+c\b(\log\b^{-1})^{-1}(\s_1-\s_2)-3.
\quad\footnote{One may take $c=\a/(10^7\log\a^{-1})$ here.}
\]
\end{thm}

The aim of this procedure is to show that if $\l$ satisfies the hypothesis in Theorem \ref{th:main}, then
\begin{equation}\label{eq:decay0}
H(\mu_\l;2^{-n}|2^{-n+1})\ge 1-C n^{-2}
\end{equation}
for some constant $C$ depending only on $\l$.
Summing these inequalities, we find that $H(\mu_{\l};2^{-n})\ge n-C$
(for some other constant $C$), which will be enough to conclude that $\mu_\l$ is absolutely
continuous.

This strategy is very similar to the one pursued by Hochman in \cite{hochman}, however,
there are very crucial differences in the details.
In comparison, Hochman's argument prove (under milder conditions)
\[
H(\mu_\l;2^{-n}|2^{-n+1})\to 1,
\]
which is enough to conclude that $\dim\mu_\l=1$.
One of the new contributions of this paper enabling the estimate on the speed of convergence
in \eqref{eq:decay0}
are the new quantitative estimates in Theorems \ref{th:high-entropy-convolution}
and \ref{th:low-entropy-convolution} for the growth of entropy under convolution.
Theorem \ref{th:low-entropy-convolution} is closely related to \cite{hochman}*{Theorem 2.7},
but the crucial difference is that the entropy growth is quantified in terms of the parameters in
Theorem \ref{th:low-entropy-convolution}.
In addition, the proofs of Theorems  \ref{th:high-entropy-convolution}
and \ref{th:low-entropy-convolution} follow a different strategy.
One of the new features of our proofs is that we estimate the entropy of the
convolution of two measures directly, which brings significant quantitative improvements.
This is in contrast with \cite{hochman}
(and \cite{Bou-discretized2}, where the size of sum-sets is estimated instead of the entropy of convolutions),
where the entropy of the convolution product of a large number of measures is estimated first
and then Pl\"unnecke-type inequalities are used.

The way the measure $\mu_\l$ is decomposed as a convolution of measures is also
new and has been optimized to achieve fast speed of convergence.

Theorem \ref{th:high-entropy-convolution} says that the missing entropy of a convolution
is the square (up to a logarithmic loss) of the missing entropy of the factors.
The exponent $2$ here is optimal, and crucial for the success of our strategy, because
it yields a polynomial decay in \eqref{eq:decay0}.

Indeed, the following simple example shows that both Theorems \ref{th:high-entropy-convolution}
and \ref{th:low-entropy-convolution} are optimal up to logarithmic factors.
Consider the measures $\mu=(1-\a)\chi_{[0,1]}+\a\d_0$ and $\nu=(1-\b)\chi_{[0,1]}+\b\d_0$,
where $\chi_{[0,1]}$ is the Lebesgue measure restricted to the interval $[0,1]$ and $\d_0$ is
the unit mass supported at $0$.
We leave it to the reader to verify that
\begin{align*}
\lim_{r\to0}H(\mu;r|2r)=&1-\a,\quad\lim_{r\to0}H(\nu;r|2r)=1-\b,\\
\lim_{r\to0}H(\mu*\nu;r|2r)=&1-\a\b.
\end{align*}

However, Theorem \ref{th:high-entropy-convolution} is useful only when the missing entropy is
smaller than a very small absolute constant.
We need to use Theorem \ref{th:low-entropy-convolution} when the measures have large missing
entropy on most levels.
The quantitative aspects of this theorem is responsible for the constraints we need to impose in
Theorem \ref{th:main} on $\l$ to obtain the exponent $2$ on the right hand side of \eqref{eq:decay0}.

Theorem \ref{th:low-entropy-convolution} is also well adapted for studying the set
of parameters $\l$ such that $\dim\mu_\l=1$.
If $\l$ is such that $\dim\mu_\l<1$, then $H(\mu^I;r|2r)<1-\a$ for all $I\subset \R_{>0}$
and $r\in\R_{>0}$, where $\a>0$ is a number depending on $\l$ but which is independent of
$I$ and $r$.
This means that the hypotheses of Theorem \ref{th:low-entropy-convolution} hold
for these measures.
This fact is proved in the
forthcoming paper \cite{BV-transcendent} of Breuillard and the author,
which aims at studying the set of $\l$'s where $\mu_\lambda$ has full dimension.

\subsection{Examples}\label{sc:explicit}

The constant $c$ in Theorem \ref{th:main} is effective, that is, it can be computed
explicitly by first obtaining an explicit value for the constant in \cite{BV-entropy}*{Theorem 5}
and then by following the steps of the proof
and substituting the values of the constants with explicit values throughout the calculations
in this paper.
The calculations in this paper pose no difficulty in this regard, and an explicit bound for the constant
in  \cite{BV-entropy}*{Theorem 5} only requires a lower bound for the integral
\begin{equation}\label{eq:numeric}
\int_\R g_1(x) \log g_1(x) dx - \int_\R g_{\sqrt 2}(x) \log g_{\sqrt 2}(x) dx, 
\end{equation}
where
\[
g_j(x)= p g(x+j)+(1-p) g(x-j)
\]
and $g(x)$ is the density function of the standard Gaussian random variable.
\footnote{Indeed, this integral gives a lower bound for $\Phi_{\nu_0}(\sqrt 2)$
as defined in \cite{BV-entropy}*{Section 3.2}, which in turn gives a lower bound
for $\Phi_{\nu_0}(a)$ for $a\ge\sqrt 2$, because of the monotonicity of $\Phi_{\nu_0}$
proved in \cite{BV-entropy}*{Lemma 14}.
The constant in \cite{BV-entropy}*{Theorem 5} is $\min_{\sqrt 2\le a\le 2}\Phi_{\nu_0}(a)/\log (a)$
due to \cite{BV-entropy}*{Corollary 15}.}

The integral \eqref{eq:numeric} can be estimated numerically using a computer, at least for a fixed value of $p$.
We do not pursue this, but in order to obtain explicit examples of absolutely continuous Bernoulli
convolutions, we keep track of the constants in the special case when $\l$ is not a root of any polynomial
with coefficients $-1$, $0$ and $1$.
In this special case, \cite{BV-entropy}*{Theorem 5} is not required for the proof of Theorem \ref{th:main}.
This assumption restricts generality, but it holds in many cases, e.g. when $\l$ is not a unit
(i.e. the leading or the constant coefficient of its minimal polynomial is other than $\pm1$),
or when $\l$ has a Galois conjugate of modulus $>2$ or $<1/2$.

If $\l$ is not a root of a polynomial with  coefficients $-1$, $0$ and $1$ and $1/4\le p\le 3/4$, then the proof of Theorem
\ref{th:main} yields that $\mu_{\l,p}$ is absolutely continuous
provided
\begin{equation}\label{eq:explicit-condition}
\l>1-10^{-37}(\log(M_\l+1))^{-1}(\log\log( M_\l+2))^{-3}.
\end{equation}
We give the proof of Theorem \ref{th:main} with inexplicit constants, but we keep track of the values
of the various constants and parameters in footnotes making no efforts at optimization.
A reader not interested in the explicit value of the constant $c$ in Theorem \ref{th:main} may
ignore these footnotes.

The problem of determining the relationship between the Mahler measure
and how close to $1$ an algebraic number can be has a rich literature, see Section 4.14 in \cite{Smy-Mahler-survey}.
We now discuss a couple of simple constructions that allow us to find examples when
Theorem \ref{th:main} applies.

\subsubsection{Rational numbers}

If $\l=1-a/b$ is a rational number for some integers $a,b\in\Z_{>0}$, then $M_\l\le b$.
(If $a$ and $b$ are coprime, then $M_\l=b$.)
This means that $\mu_{1-a/b,p}$ is absolutely continuous provided $1/4\le p\le 3/4$ and
\[
0<a<10^{-37}\frac{b}{\log (b+1)(\log\log (b+2))^3}.
\]

\subsubsection{High degree roots of algebraic numbers}

Let $\l\in(0,1)$ be an algebraic number and let $k\in\Z_{>0}$.
Then
\begin{equation}\label{eq:kthroot}
M_{\l^{1/k}}= M_\l^{\deg(\l^{1/k})/k\deg(\l)},
\end{equation}
where $\deg(x)$ denotes the degree of the number field $\Q(x)$.
Indeed, $\ln M_x=\deg(x) h(x)$ for any algebraic number $x$, where
$h(x)$ denotes the absolute logarithmic height of $x$, see \cite{BG-heights}*{Proposition 1.6.6},
and $h(x^a)=|a|h(x)$ for any $a\in\Q$, see \cite{BG-heights}*{Lemma 1.5.18}.

Clearly $\deg(\l^{1/k})\le k\deg(\l)$, hence $M_{\l^{1/k}}\le M_\l$ always.

We can also get a lower bound.
Let $\l_1=\l^{1/k},\l_2,\ldots,\l_d$ be the roots of the minimal polynomial of $\l^{1/k}$ over the field $\Q(\l)$.
Then $d=\deg(\l^{1/k})/\deg(\l)$.
We clearly have $\l_j^k=\l$ for any $j=1,\ldots, d$, hence $\l_j=\l^{1/k}\zeta_j$, where $\zeta_j$ is a root of unity
and
\[
\l_1\cdots\l_d=\l^{d/k}\cdot \zeta_2\cdots\zeta_d\in \Q(\l).
\]
Now $ \zeta_2\cdots\zeta_d$ is a root of unity, and it is real, since both $\l^{d/k}$ and $\Q(\l)$ are real.
Then $ \zeta_2\cdots\zeta_d=\pm1$, hence $\l^{d/k}\in\Q(\l)$.
Applying \eqref{eq:kthroot} with $\l^{d/k}$ in place of $\l$  and $d$ in place of $k$, we get
\begin{align*}
M_{\l^{1/k}}=&M_{\l^{d/k}}^{\deg(\l^{1/k})/d\deg(\l^{d/k})}
=M_{\l^{d/k}}^{\deg(\l)/\deg(\l^{d/k})}\\
\ge& M_{\l^{d/k}}\ge \min_{x\in\Q(\l): M_x>1} M_x>1.
\end{align*}
The existence of the minimum follows from Northcott's theorem \cite{BG-heights}*{Theorem 1.6.8}.
An effective bound for the minimum in terms of $\deg(\l)$ can be obtained from  
\cite{dobrowolski}*{Theorem 1}.

This shows that $M_{\l^{1/k}}$ stays bounded away from both $1$ and $\infty$, however,
$\l^{1/k}\to 1$ as $k$ grows.
By Theorem \ref{th:main}, we see that $\mu_{\l^{1/k}, p}$ is absolutely continuous
for any fixed algebraic $\l\in(0,1)$ and for any $p\in(0,1)$ provided $k$ is sufficiently large
depending only on $\l$ and $p$.

If $\l$ is not a unit, then the explicit bound \eqref{eq:explicit-condition} can be applied.
In particular, we have that $\mu_{n^{-1/k},p}$ is absolutely continuous for any integers $n,k\in\Z_{>0}$
provided $p\in[1/4,3/4]$ and
\[
n^{-1/k}>1-10^{-37}(\log (n+1))^{-1}(\log\log (n+2))^{-3}.
\]
Using $n^{-1/k}>1-\ln(n)/k$ we can rewrite the above condition as
\[
k>10^{37}\ln(n)\log(n+1)(\log\log (n+2))^{3}.
\]

\subsubsection{Polynomials with few non-zero coefficients}

There are good bounds on the Mahler measure in terms of the coefficients of the minimal polynomial.
Let $\l$ be an algebraic number with minimal polynomial $P(x)=a_dx^d+\ldots+a_0\in\Z[x]$.
For $q\ge 1$, we put $\ell_q(P)=(|a_d|^q+\ldots+|a_0|^q)^{1/q}$ and $\ell_\infty=\max(|a_d|,\ldots,|a_0|)$.
Then we have the bounds \cite{BG-heights}*{Lemma 1.6.7}.
\begin{equation}\label{eq:Mahler-upper}
M_\l\le \ell_2(P)\le \min(\ell_1(P),(d+1)^{1/2}\ell_{\infty}(P))
\end{equation}

This means that it is easy to construct families of polynomials whose Mahler measure stays bounded;
indeed, this holds if the polynomial has a bounded number of non-zero coefficients
that are also bounded.

We describe one possible method to construct such polynomials with roots near $1$.
Fix a polynomial $Q$ of degree $d$ such that its coefficients satisfy $2|a_j$ for all $j$ but $4\nmid a_0$ and $Q(1)<0$.
Let $n$ be a large integer and let $x_0>1$ be a root of the polynomial $x^n+Q(x)$. 
By the Eisenstein criterion, this is an irreducible polynomial, and by $|a_0|>1$, $x_0$
is not a unit, hence \eqref{eq:explicit-condition} can be applied for $\l=x_0^{-1}$.

We estimate $1-\l^{-1}<x_0-1$.
To this end, we write
\[
0=x_0^n + Q(x_0)\ge x_0^n-x_0^d\ell_1(Q),
\]
hence $x_0\le \ell_1(Q)^{1/(n-d)}$.
In light of \eqref{eq:explicit-condition}
$\mu_{\l,p}$ is absolutely continuous, provided $p\in[1/4,3/4]$ and $n$
is sufficiently large so that
\[
\ell_1(Q)^{1/(n-d)}<1+10^{-37}(\log(\ell_1(Q)+1)\log\log(\ell_1(Q)+2))^{-3}.
\]

The restrictive hypothesis on the coefficients of $Q$ is not necessary, it is only required
to ensure that the polynomial is irreducible and \eqref{eq:explicit-condition} can be applied.
However, the upper bound \eqref{eq:Mahler-upper} is still valid for any root $\l$ of $P$ even
if $P$ is not irreducible.
A general lower bound
\[
M_\l>1+\frac{1}{1200}\Big(\frac{\log\log n}{\log n}\Big)^3
\]
due to Dobrowolski \cite{dobrowolski}*{Theorem 1} is available for any algebraic number of degree at most $n$ that is not a root of unity.
Since $\frac{1}{1200}\Big(\frac{\log\log n}{\log n}\Big)^3$ decreases much slower than $\ell_1(Q)^{1/(n-d)}-1$ as $n$ grows,
Theorem \ref{th:main} implies that $\mu_{\l,p}$ is absolutely continuous if $n$ is sufficiently large for any fixed polynomial $Q$
and fixed $p\in(0,1)$
assuming only that $Q(1)<0$.

\subsubsection{More general considerations}

Since the right hand side of \eqref{eq:explicit-condition} approaches $1$ very slowly as $M_\l$ grows,
it is not necessary to restrict our attention to examples with bounded Mahler measure.
Indeed, for the Mahler measure of a polynomial $P$, we have
\begin{equation}\label{eq:lge1}
M_P\ge \prod_{\l:P(\l)=0,|\l|>1} |\l|.
\end{equation}
(We have equality if $P$ is monic.)
Hence a ``typical'' root $\l$ of $P$ is expected to satisfy $|\l|<M_P^{c/\deg(P)}$,
which approaches $1$ much more rapidly than the right hand side of \eqref{eq:explicit-condition}.
Of course, a ``typical'' root is not expected to be real, but the following construction can be used
to find more examples, for which the theorem applies.

Let $P\in\Z[x]$ be an irreducible polynomial that is not reciprocal, i.e. $a_{d-i}\neq a_i$ for at least one $i$,
where $d=\deg(P)$ and $a_i$ are the coefficients of $P$.
Such polynomials can be found in abundance using Eisenstein's criterion.
These assumptions imply that $P$  has no roots on the unit circle.
For simplicity, assume that $P$ has at least $d/2$ roots outside the unit circle, otherwise simply
replace $P$ by $x^dP(x^{-1})$.

By \eqref{eq:lge1}, there is a root $x_0$ of $P$ such that
$1<|x_0|<M_P^{2/d}$.
We take $\l=|x_0|^{-2}$ and observe that
\[
\l>1-4\ln(M_P)/d.
\]
It is easy to see from the definition that $M_\l\le M_P^2$, hence $\mu_{\l,p}$ is absolutely continuous
provided $1/4\le p \le3/4$ and
\[
4\ln(M_P)/d<10^{-37}(\log(M_P^2+1))^{-1}(\log\log(M_P^2+2))^{-3}.
\]
It is straightforward to find such polynomials using \eqref{eq:Mahler-upper}.

\subsection{Notation}

We denote by the letters $c$, $C$ and their indexed variants various constants
that could in principle be computed explicitly following the proof step by step.
The value of these constants denoted by the same symbol may change between occurrences.
We keep the convention that we denote by lower case letters the constants that are
best thought of as ``small'' and by capital letters the ones that are ``large''.

We keep track of the values of these constants in footnotes.
A reader not interested in the explicit values of these constants may choose to ignore these footnotes.
A footnote of the form $C_j=x$ means that the constant $C$ in the line where the footnote points
should be substituted by $C_j$ and its value can be taken $x$.
This is usually followed by an explanation or detailed calculation.
The constants are indexed in a manner that ensures that the value of the constant with index $j$ may
depend only on the constants with index less than $j$.
In Section \ref{sc:LER} the values given for the constants are valid under the additional hypothesis
that $\l$ is not the root of a polynomial with coefficients $-1$, $0$ and $1$ and $1/4\le p\le 3/4$.
In other parts of the paper, no hypothesis is required beyond those stated in the main body of the text.

We denote by $\log$ the base $2$ logarithm and write $\ln$ for the  logarithm in base $e$.

\subsection{The organization of this paper}

We begin by discussing some basic properties of entropy in Section \ref{sc:entropy}, which
we will rely on throughout the paper.
Sections \ref{sc:proof-high-ent-conv} and \ref{sc:low-entropy-convolution} are devoted to the proofs
of Theorems \ref{th:high-entropy-convolution} and \ref{th:low-entropy-convolution}
respectively.
We conclude the paper in Section \ref{sc:proof}
by explaining the details of the argument outlined above to prove
Theorem \ref{th:main}.

\subsection*{Acknowledgment}

I am indebted to Elon Lindenstrauss with whom we discussed entropy increases under convolutions
\cite{LV-entropy-sum-product} for several years in connection with Bourgain's discretized sum product theorem.
I am also indebted to Emmanuel Breuillard with whom we studied the quantity $h_\l$ in
\cite{BV-entropy}.
These works have been a rich source of inspiration for this project.
I am grateful to Mike Hochman for useful discussions, and in particular, for suggesting
to consider biased Bernoulli convolutions.

I am grateful to the anonymous referees for carefully reading my paper and for numerous
helpful remarks and suggestions that greatly improved the presentation.

\section{Basic properties of entropy}\label{sc:entropy}

The purpose of this section is to provide some background material on entropy.

\subsection{Shannon and differential entropies}
If $X$ is a discrete random variable, we write $H(X)$ for its Shannon entropy, that is
\[
H(X)=\sum_{x\in\cX} -\P(X=x)\log\P(X=x),
\]
where $\cX$ denotes the set of values $X$ takes.
We recall that the base of $\log$ is $2$ throughout the paper.
If $X$ is an absolutely continuous random variable with density $f:\R\to\R_{\ge0}$, we write
$H(X)$ for its differential entropy, that is
\[
H(X)=\int -f(x)\log f(x) dx.
\]
This dual use for $H(\cdot)$ should cause no confusion, as it will be always clear from the context,
what the type of the random variable is.
If $\mu$ is a probability measure, we write $H(\mu)=H(X)$, where $X$ is a random variable with law
$\mu$.

Shannon entropy is always non-negative.
Differential entropy on the other hand can take negative values.
For example, if $a \in\R_{>0}$, and $X$ is a random variable
with finite differential entropy $H(X)$, then it follows from the change of variables formula that
\begin{equation}\label{change}
H(aX) = H(X) + \log a,
\end{equation}
which can take negative values when $a$ varies.
On the other hand, if $X$ takes countably many values, the Shannon entropy of $aX$ is the same as that of $X$.
Note that both kinds of entropy are invariant under translation by a constant in $\R$,
that is $H(X)=H(X+a)$.

We define $F(x):=-x\log (x)$ for $x>0$ and recall that
$F$ is concave,
and it is sub-additive, i.e. $F(x+y) \leq F(x)+F(y)$, and it also
satisfies the identity $F(xy)=xF(y)+yF(x)$.

{}From the concavity of $F$ and Jensen's inequality,
we see that for any discrete random variable $X$ taking at most $N$ possible different values,
\begin{equation}\label{support}
H(X) \leq \log N.
\end{equation}

Let now $X$ and $Y$ be two independent random variables in $\R$.
If both are discrete, it follows immediately from the identity $F(xy)=xF(y)+yF(x)$
and the sub-additivity of $F(x)$
that $H(X+Y) \leq H(X) + H(Y)$ for  Shannon entropy.
This is no longer true for differential entropy since the formula is not invariant under a linear change of variable.
However if $X$ is atomic and bounded, while $Y$ is assumed absolutely continuous, then
\begin{equation}\label{subadd}
H(X+Y) \leq H(X) + H(Y),
\end{equation}
where $H(X)$ is Shannon's entropy and the other two are differential entropies.
To see this,  note that if $f(y)$ is the density of $Y$,
then the density of $X+Y$ is $\E(f(y-X))= \sum_i p_if(y-x_i)$, hence:
\begin{align*}
H(X+Y)&= \int F\Big(\sum_i p_if(y-x_i)\Big) dy\\
&\leq \int \sum_i F(p_i f(y-x_i)) dy\\
&= \int \sum_i F(p_i) f(y-x_i) dy + \int \sum_i p_i F(f(y-x_i)) dy \\
&= \sum_i F(p_i) + \int  F(f(y)) dy = H(X) + H(Y)
\end{align*}
and $(\ref{subadd})$ follows.

In the other direction, we always have the lower bound
\begin{equation}\label{lowerbound}
H(X+Y) \geq \max(H(X),H(Y))
\end{equation}
if all three entropies are of the same type (i.e. either Shannon or differential), as follows easily from the concavity of $F$.

Let $X$ and $Y$ be two discrete random variables.
We define the conditional entropy of $X$ relative to $Y$ as
\begin{align*}
H(X|Y)=&\sum_{y\in\cY} \P(Y=y) H(X|Y=y)\\
=&\sum_{y\in\cY} \P(Y=y) \sum_{x\in\cX} -\frac{\P(X=x,Y=y)}{\P(Y=y)}\log\frac{\P(X=x,Y=y)}{\P(Y=y)}.
\end{align*}
We recall some well-known properties.
We always have $0\le H(X|Y)\le H(X)$, and $H(X|Y)=H(X)$ if and only if the two random variables are
independent (see \cite{cover-thomas}*{Theorem 2.6.5}).
The entropy of the joint law can be expressed as
\[
H(X,Y)=H(X|Y)+H(Y)\le H(X)+H(Y).
\]
If $f$ is any function defined on $\cX$, we have $H(f(X)|X)=0$ as seen from the definition.
This implies $H(f(X))\le H(X)$.
In particular, taking $f(x,y)=x+y$ and applying the above inequality for the joint distribution
of the random variables $X$ and $Y-X$, we obtain
\[
H(Y)\le H(X,Y-X)\le H(X)+H(Y-X).
\]
By reversing the roles of $X$ and $Y$ we get
\begin{equation}\label{eq:entropy-difference}
|H(X)-H(Y)|\le H(Y-X).
\end{equation}

We also note the identity
\[
H(X|f(X))=H(X)-H(f(X)).
\]
We will use this repeatedly in what follows.

We recall the following result from \cite{madiman}*{Theorem I}.

\begin{prp}[Submodularity inequality]
\label{ruzsa}
Assume that $X,Y,Z$ are three independent $\R$-valued random variables
such that the distributions of $Y$, $X+Y$, $Y+Z$ and $X+Y+Z$ are absolutely continuous with respect
to Lebesgue measure and have finite differential entropy.
Then
\begin{equation}\label{ruzsaineq}
H(X+Y+Z) + H(Y) \leq H(X+Y) + H(Y+Z).
\end{equation}
\end{prp}

This result goes back in some form at least to a
paper by Kaimanovich and Vershik \cite{kaimanovich-vershik}*{Proposition 1.3}.
The version in that paper assumes that the laws of $X$, $Y$ and $Z$ are identical.
The inequality was rediscovered by Madiman \cite{madiman}*{Theorem I} in the
greater generality stated above.
Then it was recast in the context of entropy analogues of sumset
estimates from additive combinatorics by  Tao \cite{tao}
and Kontoyannis and Madiman \cite{kontoyannis-madiman}.
And indeed Theorem \ref{ruzsa} can be seen as an entropy analogue
of the Pl\"unnecke--Ruzsa inequality in additive combinatorics.
For the proof of this exact formulation see \cite{BV-entropy}*{Theorem 10}.

\subsection{Entropy at a given scale}\label{sc:entropy-at-scale}

We recall the notation
\[
H(X;r)=\int_{0}^1 H(\lfloor X/r+t\rfloor) dt
\]
and
\[
H(X;r_1|r_2)=H(X;r_1)-H(X;r_2).
\]

These quantities originate in the work of Wang \cite{Wan-quantitative}, and they also
play an important role in the paper \cite{LV-entropy-sum-product}, where a
quantitative version of Bourgain's sum-product
theorem is proved.

We continue by recording some useful facts about these notions.
We note that entropy at scales have the following scaling property
\begin{equation}\label{eq:scale-scaling}
H(X;r)=H(sX;sr)
\end{equation}
for any numbers $s,r>0$, which follows immediately from the definition.

If $N$ is an integer then we have the following interpretation, which follows easily from the definition
\begin{equation}\label{equation:interpret2}
H(X;N^{-1}r|r)=\int_0^1 H\big(\lfloor N(r^{-1}X+t)\rfloor\big|\lfloor r^{-1} X+t\rfloor\big)dt.
\end{equation}
Indeed, $\lfloor r^{-1} X+t\rfloor$ is a function of $\lfloor N(r^{-1}X+t)\rfloor$, hence
\[
H\big(\lfloor N(r^{-1}X+t)\rfloor\big|\lfloor r^{-1} X+t\rfloor\big)
=H(\lfloor N(r^{-1}X+t)\rfloor)-H(\lfloor r^{-1} X+t\rfloor).
\]

Combining \eqref{equation:interpret2} and \eqref{support} we see that the entropy between
scales of ratio $N\in \Z$ is at most $\log N$.
It is not difficult to see that this upper bound is sharp
(consider uniform measures on very long intervals), though equality is never attained.

The next lemma gives an alternative definition for entropy at a given scale.
\begin{lem}\label{lm:second-def}
Let $X$ be a bounded random variable in $\R$.
Then
\[
H(X;r)=H(X+I_r)-H(I_r)=H(X+I_r)-\log(r).
\]
where $I_r$ is a uniform random variable in $[0,r]$ independent of $X$.
\end{lem}
\begin{proof}
By \eqref{change} and \eqref{eq:scale-scaling}, both sides of the identity
are scaling invariant,
hence it is enough to prove the lemma for $r=1$.
Let $\mu$ be the distribution of $X$.
Then the density of $X+I_1$ is $\mu[x-1,x)$, hence
\begin{align*}
H(X+I_1)=&\int_{-\infty}^{\infty} F(\mu([x,x+1)))dx
=\int_{0}^1\sum_{n\in\Z}F(\mu([n-t,n-t+1)))dt\\
=&\int_{0}^1\sum_{n\in\Z}F(\P(\lfloor X+t\rfloor=n))dt
=\int_0^1 H(\lfloor X+t\rfloor)dt=H(X;1).
\end{align*}
\end{proof}

It follows from the definition that being an average of Shannon entropies
$H(X;r)$, is always non-negative.
Similarly, we see from \eqref{equation:interpret2} that $H(X;r_1|r_2)$ is also non-negative
if $r_2/r_1$ is an integer.
We will see below that this holds also for any $r_2\ge r_1$.

We show that conditional entropy between scales of integral ratio
cannot decrease by taking convolution of measures.
\begin{lem}\label{lm:entropy-conv-nondecrease}
Let $X$ and $Y$ be two bounded independent random variables in $\R$.
Let $r_2>r_1>0$ be two numbers such that $r_2/r_1\in\Z$.
Then
\[
H(X+Y;r_1|r_2)\ge H(X;r_1|r_2).
\]
\end{lem}

\begin{proof}
Write $I_{r_2}=I_{r_1}+Z$, where
$I_{r_i}$ are uniform random variables on $[0,r_i]$ for $i=1,2$  and $Z$
is uniformly distributed on the arithmetic progression $\{0,r_1,2r_1,\ldots,r_2-r_1\}$ and is independent
of $I_{r_1}$.
Now using submodularity (Proposition \ref{ruzsa}), we can write
\begin{align*}
H(X+Y;r_1|r_2)=H(X+Y+I_{r_1})-H(X+Y+I_{r_1}+Z)+\log(r_2/r_1)\\
\ge H(X+I_{r_1})-H(X+I_{r_1}+Z)+\log(r_2/r_1)
=H(X;r_1|r_2).
\end{align*}
\end{proof}

It is reasonable to expect that perturbation on a small scale does not affect
entropy at a much larger scale.
A particular instance of this is proved in the next lemma.

\begin{lem}\label{lm:perturb}
Let $r_1,r_2$ be two positive real numbers such that $2r_1\le r_2$.
Let $X$ and $Y$ be two random
variables such that $X\le Y\le X+r_1$ almost surely.
Then
\[
|H(X;r_2)-H(Y;r_2)|\le 2 \frac{r_1}{r_2}\log(r_2/r_1).
\]
\end{lem}

\begin{proof}
We have
\[
H(X;r_2)=\int_0^1H(\lfloor r_2^{-1}X+t\rfloor ) dt,\quad
H(Y;r_2)=\int_0^1H(\lfloor r_2^{-1}Y+t\rfloor ) dt.
\]
Hence using \eqref{eq:entropy-difference} we can write
\begin{align*}
|H(X;r_2)-H(X;r_2)|
\le&\int_0^1|H(\lfloor r_2^{-1}X+t\rfloor )-H(\lfloor r_2^{-1}Y+t\rfloor )|dt\\
\le&\int_0^1H(\lfloor r_2^{-1}Y+t\rfloor -\lfloor r_2^{-1}X+t\rfloor )dt.
\end{align*}

We note that
$\lfloor r_2^{-1}Y+t\rfloor -\lfloor r_2^{-1}X+t\rfloor$
is equal to $0$ or $1$ almost surely,
since $X\le Y\le X+r_1$ and $r_1\le r_2$.

For all $x,y\in \R$ with $x\le y\le x+r_1$ we have
\[
\int_0^1 \lfloor r_2^{-1}y+t\rfloor -\lfloor r_2^{-1}x+t\rfloor dt\le \frac{r_1}{r_2}.
\]
Thus
\begin{align*}
\int_0^1\P(\lfloor r_2^{-1}Y+t\rfloor -&\lfloor r_2^{-1}X+t\rfloor=1)dt\\
=&\E\Big[\int_0^1\lfloor r_2^{-1}Y+t\rfloor -\lfloor r_2^{-1}X+t\rfloor dt \Big]
\le \frac{r_1}{r_2}.
\end{align*}

For each $0\le t\le 1$ we have
\[
H(\lfloor r_2^{-1}Y+t\rfloor -\lfloor r_2^{-1}X+t\rfloor)
= h(\P(\lfloor r_2^{-1}Y+t\rfloor -\lfloor r_2^{-1}X+t\rfloor=1)),
\]
where $h(x)=-x\log x-(1-x)\log(1-x)$.

By Jensen's inequality we then have
\[
\int_0^1H(\lfloor r_2^{-1}Y+t\rfloor -\lfloor r_2^{-1}X+t\rfloor )dt\le h(r_1/r_2).
\]
This proves the claim, since $h(x)\le- 2x\log x$ for $x\le1/2$.
\end{proof}

In the next lemma we show that $H(X;r)$ is a monotone increasing and Lipschitz
function of $-\log r$; in particular $H(X;r_1|r_2)$ is nonnegative for all $r_1\le r_2$.
The lemma is taken from \cite{LV-entropy-sum-product}, but we include the proof
for the reader's convenience.

\begin{lem}\label{lemma:Lipschitz}
Let $X$ be a bounded random variable in $\R$.
Then for any $r_1\ge r_2>0$ we have
\[
0\le H(X;r_2)-H(X;r_1)\le 2(\log r_1-\log r_2).
\]
\end{lem}
\begin{proof}
We observe that the density of $X+rI_1$ is equal to $r^{-1}\P(X\in[x-r,x])$.
Then
\begin{align*}
H(X&;r)=H(X+rI_1)-\log r\\
=&-\int_{\R}r^{-1}\P(X\in[x-r,x])\log(r^{-1}\P(X\in[x-r,x]))dx-\log r\\
=&\int_\R\int_0^1-\log(\P(X\in[y+rt-r,y+rt]))dtd\mu(y),
\end{align*}
where we substituted $x=y+rt$ and used again the fact that $r^{-1}\P(X\in[x-r,x])$
is the density of $X+rI_1$.
We note that the function $-\log(\P(X\in[y+(t-1)r,y+tr]))$ is an increasing function of $-\log r$
for any fixed $y$ and $t\in[0,1]$,
hence the lower bound follows.

For the upper bound, we assume without loss of generality that $X$ is absolutely continuous
(use e.g. Lemma \ref{lm:perturb}) and write $f$ for its density.
In this case
\[
\frac{d}{dr}\P(X\in[y+r(t-1),y+rt])=(1-t)f(y+r(t-1))+tf(y+rt).
\]
Hence
\begin{align*}
\left.\frac{dH(X;r)}{dr}\right|_{r=1}
=&-\int_{\R}\int_{0}^1\frac{(1-t)f(y+(t-1))+tf(y+t)}{(\ln2)\P(X\in[y+(t-1),y+t])}dtf(y)dy\\
\ge&-\int_{\R}\int_{y}^{y+1}\frac{(f(z-1)+f(z))f(y)}{(\ln2)\P(X\in[z-1,z])}dzdy\\
=&-\int_{\R}\int_{z-1}^z\frac{(f(z-1)+f(z))f(y)}{(\ln2)\P(X\in[z-1,z])}dydz\\
=&-2(\ln2)^{-1}.
\end{align*}
To derive the second line, we used the estimates $t\le1$, $1-t\le1$
and the substitution $z=y+t$.
{}From this we conclude
\[
\left.\frac{dH(X;2^{-\rho})}{d\rho}\right|_{\rho=0}\le 2.
\]
Dilating $X$, we derive the same inequality for all $\rho$.
This implies the upper bound in the lemma.
\end{proof}

We have seen above that convolution can only increase entropy between scales of integral
ratio.
Unfortunately, this does not hold for general scales, but it does hold with small error,
provided the ratio of the scales is large.
This is the content of the next lemma.
Recall that we write $H(\mu)=H(X)$, when the measure $\mu$ is the law of the random variable $X$.

\begin{lem}\label{lm:fractional-scale}
Let $\mu$ and $\nu$ be two compactly supported probability measures on $\R$
and be $0<r_2<r_1$ numbers.
Then
\[
H(\mu*\nu;r_2|r_1)\ge H(\mu;r_2|r_1)-\frac{2}{(\ln 2)(r_1/r_2-1)}.
\]
\end{lem}

\begin{proof}
Write $N=\lfloor r_1/r_2\rfloor$.
Then 
\[
H(\mu*\nu;r_2|Nr_2)\ge H(\mu;r_2|N r_2)
\]
by our previous discussion.

By Lemma \ref{lemma:Lipschitz}, we have
\begin{align*}
H(\mu*\nu;Nr_2|r_1)\ge& 0,\\
H(\mu;Nr_2|r_1)\le& 2\log \frac{r_1}{Nr_2}.
\end{align*}

Combining our estimates, we find
\[
H(\mu*\nu;r_2|r_1)\ge H(\mu;r_2|r_1)-2\log \frac{r_1}{Nr_2}.
\]
We note $Nr_2\le r_1< (N+1)r_2$, hence
\[
 1\le \frac{r_1}{Nr_2}<1+\frac{1}{N}<1+\frac{1}{r_1/r_2-1}.
\]
Now the claim follows from $\log (1+x)\le(\ln 2)^{-1} x$.
\end{proof}

\subsection{Entropy of non-probability measures} \label{sc:non-prob}

It is convenient to use the notation $H(\mu)$, $H(\mu;r)$ and $H(\mu;r_1|r_2)$
for Shannon and differential entropies and for entropies at a scale
also for positive measures $\mu$ that have total mass different from $1$.
Let $\mu$ be such a measure and write $\|\mu\|$ for its total mass.
In this paper, we use the conventions
\begin{equation}\label{eq:non-probability}
H(\mu)=p H(p^{-1}\mu),\quad 
H(\mu;r)=p H(p^{-1}\mu;r),
\end{equation}
where $p=\|\mu\|$.

With this convention, entropy has the following superadditive property.
Let $\mu,\ldots,\mu_n$ and $\mu$ be positive measures of the same type and $a_1,\ldots, a_n$
positive real numbers such that
$\mu=a_1\mu_1+\ldots+a_n\mu_d$.
Then
\begin{equation}\label{eq:superadditive}
H(\mu)\ge a_1 H(\mu_1)+\ldots+a_n H(\mu_d)
\end{equation}
holds for both Shannon and differential entropies.
If all the measures are probabilities, then this is an immediate consequence
of Jensen's inequality applied to the concave function $F(x)=-x\log x$.
The general case follows from this special case and the convention \eqref{eq:non-probability}.

Entropies at scales are also superadditive, since they are defined as
averages of Shannon entropies.
Moreover, this property also holds for conditional entropies between scales
of integral ratio.

\begin{lem}\label{lm:conditional-super-additive}
Let $\mu_1,\ldots,\mu_k$ be non-negative compactly supported measures on $\R$,
$r>0$ and $N\in\Z_{>0}$.
Then
\[
H(\mu_1+\ldots+\mu_k;N^{-1}r|r)\ge H(\mu_1;N^{-1}r|r)+\ldots + H(\mu_k;N^{-1}r|r).
\]
\end{lem}

\begin{proof}
By \eqref{eq:scale-scaling},
we may assume without loss of generality that $r=N$.

For a random variable $X$ with law $\mu$, we have the formula
\eqref{equation:interpret2}:
\[
H(\mu;1|N)=H(X;1|N)=\int_0^1 H\big(\lfloor X+Nt\rfloor\big|\lfloor N^{-1} X+t\rfloor\big)dt.
\]

For each $t\in[0,1]$ and $a\in\Z$, we define the non-negative measure $\rho_{t,a}$ on $[0,N-1]\cap\Z$
by
\[
\rho_{t,a}(j)=\P(\lfloor X+Nt\rfloor=aN+j)=\mu([aN-tN+j,aN-tN+j+1)).
\]

Using these measures, the definition of conditional entropy reads
\[
H\big(\lfloor X+Nt\rfloor\big|\lfloor N^{-1} X+t\rfloor\big)
=\sum_{a\in\Z} \|\rho_{t,a}\| H(\|\rho_{t,a}\|^{-1}\rho_{t,a})
=\sum_{a\in\Z} H(\rho_{t,a}).
\]

We plug this in our first formula and obtain:
\[
H(X;1|N)=\int_0^1 \sum_{a\in\Z} H(\rho_{t,a}) dt.
\]

Therefore, we can express conditional entropy between scales of integral ratio
as an integral of a sum of Shannon entropies.
Hence the lemma reduces to superadditivity of Shannon entropies. 
\end{proof}

\subsection{\texorpdfstring{Measures supported on $\Z$}{Measures supported on Z}}

We consider measures supported on $\Z$ in this section, and develop some formulae
for their entropies.
Let $X$ be an integer valued random variable, and let $M\in\Z_{>1}$.

Using the formula \eqref{equation:interpret2}, we can write
\[
H(X;1|M)=\int_0^1 H\big(\lfloor X+Mt\rfloor\big|\lfloor M^{-1}X+t\rfloor\big)dt.
\]
We observe that the integrand is constant on the interval $[aM^{-1},(a+1)M^{-1})$
and is equal to $H\big( X+a\big|\lfloor M^{-1}(X+a)\rfloor\big)$, for each $a=0,1,\ldots, M-1$.
Hence we can write
\begin{equation}\label{eq:ent-on-Z1}
H(X;1|M)=\frac{1}{M}\sum_{a=0}^{M-1}H\big(X+a\big|\lfloor M^{-1}(X+a)\rfloor\big).
\end{equation}

For each $a\in\Z$, write $\rho_a$ for the restriction of the law of $X$ to the
interval $[a,a+M-1]$ without normalization, that is
$\rho_a(n)=\P(X=n)$ for $n\in \{a,\ldots,a+M-1\}$ and $\rho_a(n)=0$ otherwise.
We note that
\begin{align*}
H\big(X+a&\big|\lfloor M^{-1}(X+a)\rfloor\big)\\
=&\sum_{b\in\Z}\P(\lfloor M^{-1}(X+a)\rfloor=b)H\big(X+a\big|\lfloor M^{-1}(X+a)\rfloor=b\big)\\
=&\sum_{b\in\Z}\|\rho_{Mb-a}\|\cdot H(\|\rho_{Mb-a}\|^{-1}\cdot\rho_{Mb-a})\\
=&\sum_{b\in\Z}H(\rho_{Mb-a})
\end{align*}
using the convention for entropies of non-probability measures we made in the previous section.
We combine this with \eqref{eq:ent-on-Z1} and get
\begin{equation}\label{eq:ent-on-Z2}
H(X;1|M)=\frac{1}{M}\sum_{a\in\Z}H(\rho_a).
\end{equation}

\section{Entropy of convolutions in the high entropy regime}
\label{sc:proof-high-ent-conv}

This section is devoted to the proof of Theorem  \ref{th:high-entropy-convolution},
which we restate.
\begin{thm*}
There is an absolute constant $C>0$ such that the following holds.
Let $\mu,\wt\mu$ be two compactly supported probability measures on $\R$ and let $0<\a<1/2$ and $r>0$ be real numbers.
Suppose that
\[
H(\mu;s|2s)\ge 1-\a\quad \text{and}\quad H(\wt\mu;s|2s)\ge 1-\a
\]
for all $s$ with $|\log r -\log s|<3\log \a^{-1}$.

Then
\[
H(\mu*\wt\mu;r|2r)\ge 1- C(\log \a^{-1})^3 \a^2.
\]
\end{thm*}

We begin with a discussion motivating the argument.
As we already noted, the entropy of a measure between the scales $1$ and $M$
for some $M\in\Z_{>0}$, is always bounded above by $\log M$.
Motivated by this, we refer to the quantity
$\log M-H(\mu;1|M)$ as the missing entropy between these scales.
If $a$, $b$ are positive integers such that $a|b$ and $b|M$, then
\[
\log(b/a)- H(\mu;a|b)\le \log M - H(\mu;1|M),
\]
that is, the missing entropy between the scales
$a$ and $b$ is at most as much as between the scales $1$ and $M$.
This follows easily from $H(\mu;1|a)\le \log a$ and $H(\mu;b|M)\le \log (M/b)$.

In this language, Theorem  \ref{th:high-entropy-convolution} can be stated informally as follows.
If we take two measures whose missing entropies are small, then the missing entropy of their
convolution may be only a little larger than the product of the missing entropies of the
factors.
The intuition behind this result is that we can decompose the probability measures $\mu$
and $\wt\mu$ as a combination of the uniform distribution (on an interval, say)
plus an error term controlled by
the missing entropy in a suitable quantitative sense.
Since the convolution of a uniform measure with any measure is (close to) uniform,
the only term contributing to the missing entropy of the convolution is the convolution
of the error terms.
We will control the ``size" of this term in a suitable sense by the product of the ``sizes''
of the error terms.

We will reduce the theorem to a problem about measures supported on the set $[1,N]\cap\Z$,
where $N$ is an integer comparable to a suitable negative power of $\a$.
We will do this in two steps beginning with the following result about measures supported
on $\Z$.

\begin{prp}\label{pr:high-ent-Z}
There is an absolute constant $C>0$ such that the following holds.
Let $\nu$ and $\wt\nu$ be two probability measures on $\Z$.
Let $N$, $M$ be two positive integers such that $2|N$ and $M|N$.
Then
\begin{align*}
\log M-H(\nu*&\wt\nu;1|M)\\
\le & C\log M(\log N-H(\nu;1|N))(\log N-H(\wt\nu;1|N))\quad\footnotemark\\
&+C\frac{M\log M}{N}.\quad\footnotemark
\end{align*}
\addtocounter{footnote}{-1}
\footnotetext{$C_6=6\cdot 10^4$.\label{page:C6}}
\stepcounter{footnote}
\footnotetext{$C_7=4000$.\label{page:C7}}
\end{prp}

To see how  Theorem \ref{th:high-entropy-convolution} can be reduced to this, we assume,
as we may, that $r$ is an integer comparable to $\a^{-C}$.
We will show that $H(\mu*\wt\mu;r|2r)$ is not sensitive to perturbations on scale $1$,
so we can replace $\mu$ and $\wt\mu$ by measures $\nu$ and $\wt\nu$ supported on $\Z$.
We will then apply Proposition \ref{pr:high-ent-Z} for these perturbed measures with $M=2r$ and $N=M^2$.
We will conclude by observing that the missing entropy between the scales $M/2$ and $M$ can be bounded above
by the missing entropy between the scales $1$ and $M$.
The details of this will be given in Section \ref{sc:proof-high-ent-th}.

Proposition \ref{pr:high-ent-Z} will be proved in Section \ref{sc:proof-high-ent-Z}
by decomposing the measures $\nu$ and $\wt\nu$ as convex combinations of measures
supported on intervals of length $N$ and using the following result.

\begin{prp}\label{pr:restricted-conv}
There is an absolute constant $C>0$ such that the following holds.
Let $N$ be a positive integer and let $\mu$ and $\wt\mu$ be two probability measures
concentrated on $[1,N]\cap\Z$.
Suppose that $2|N$ and let $M|N$.
Write $\s=(\mu*\wt\mu)|_{[N/2+1,3N/2]}$.
Then
\begin{align*}
 \|\s\|_1\cdot\log M-H(\s;1|M)
\le& C\log M (\log N -H(\mu))(\log N-H(\wt\mu))\quad\footnotemark\\
&+C\frac{M\log M}{N}.\quad\footnotemark
\end{align*}
\addtocounter{footnote}{-1}
\footnotetext{$C_4=4\cdot 10^4$.\label{page:C4}}
\stepcounter{footnote}
\footnotetext{$C_5=3000$.\label{page:C5}}
\end{prp}

The key observation behind the proof of Proposition \ref{pr:restricted-conv}
is that we can decompose $\mu$ and $\wt\mu$ as sums of pairs of functions,
such that one in the pair has controlled $L^2$ distance from the uniform distribution
and the other one has controlled $L^1$ norm in terms of the missing entropy.
Then we can write $\s$ as the combination of $4$ functions, each of which
will be estimated using different methods.

The reason for restricting the convolution to the interval $[N/2+1,3N/2]$ is technical.
In a certain stage, we will show that one of the terms
contributing to the convolution does not vary too much on intervals of appropriate length.
To convert this to a bound on entropy, we need to know that the function is
not too small, which is achieved by cutting off the ends of the support.

We will explain the above mentioned decomposition in Section \ref{sc:high-decompose}.
We estimate the convolution of two functions controlled in $L^2$ in Section \ref{sc:L2}.
We estimate the convolution of two functions, one of which is controlled in $L^2$ and one of which
is controlled in $L^1$ in Section \ref{sc:L2L1}.
Then we combine these estimates in Section \ref{sc:restricted-conv} to obtain the proof
of Proposition \ref{pr:restricted-conv}.

\begin{nota}
In this section, we write $\chi_N$ for the normalized counting measure on $[1,N]\cap\Z$,
i.e. for the function $\chi_N:\Z\to \R$ given by
\[
\chi_N(x)=
\begin{cases}
\frac{1}{N}& x\in [1,N]\\
0&\text{otherwise}.
\end{cases}
\]

Most functions and measures in this section are defined on $\Z$, and their $L^p$-norms
are defined with respect to the counting measure.
It will be convenient for us to identify measures with their densities.
\end{nota}

\subsection{}
\label{sc:high-decompose}
The purpose of this section is the following decomposition of measures
of high entropy.

\begin{lem}\label{lm:decomp}
Let $N$ be a positive integer and let $\mu$ be a probability measure concentrated on $[1,N]\cap\Z$.
Then there are two non-negative functions $f,g:\Z\to \R$ such that
$\mu=f+g$ and the following estimates hold:
\begin{align*}
\|f-\chiN\|_2^2&\le2\frac{\log N-H(\mu)}{N}, \qquad \|f\|_\infty\le\frac{2}{N},\qquad \|f\|_1\le 1,\\
\|g\|_1&\le2 (\log N-H(\mu)).
\end{align*}
\end{lem}

\begin{proof}
Set
\begin{align*}
f(n)&:=
\begin{cases}
\mu(n)&\text{if $\mu(n)\le2/N$},\\
\chiN(n)&\text{otherwise},
\end{cases}\\
g(n)&:=\mu(n)-f(n).
\end{align*}

We note the inequalities
\[
x\log x\ge
\begin{cases}
(1/\ln2)(x-1)+(2-1/\ln2)(x-1)^2&\text{if $0\le x\le 2$}\\
2x-2&\text{if $x\ge2$}.
\end{cases}
\]
We substitute $x=Ny$.
If $y\le 2/N$ we obtain
\[
y\log(Ny)\ge\frac{1}{\ln2}\Big(y-\frac{1}{N}\Big)
+\Big(2-\frac{1}{\ln2}\Big)N\Big(y-\frac{1}{N}\Big)^2,
\]
and
\begin{equation}\label{eq:lower1}
\frac{\log N}{N}+y\log y\ge
\Big(\frac{1}{\ln2}-\log N\Big)\Big(y-\frac{1}{N}\Big)
+\Big(2-\frac{1}{\ln2}\Big)N\Big(y-\frac{1}{N}\Big)^2.
\end{equation}
If $y>2/N$ we obtain
\[
y\log(Ny)\ge 2\Big(y-\frac{1}{N}\Big)
\]
and
\begin{equation}\label{eq:lower2}
\frac{\log N}{N}+y\log y\ge
\Big(\frac{1}{\ln2}-\log N\Big)\Big(y-\frac{1}{N}\Big)
+\Big(2-\frac{1}{\ln2}\Big)\Big(y-\frac{1}{N}\Big).
\end{equation}

Using  \eqref{eq:lower1} if $\mu(n)\le2/N$ and \eqref{eq:lower2} if $\mu(n)>2/N$, we can write
\begin{align}
\log N-H(\mu)&=\sum_{n=1}^{N}\Big(\frac{\log N}{N}+\mu(n)\log\mu(n)\Big)\nonumber\\
&\ge\Big(\frac{1}{\ln 2}-\log N\Big)\sum_{n=1}^{N}\Big(\mu(n)-\frac{1}{N}\Big)\label{equation:linterm}\\
&\quad+\sum_{\mu(n)\le2/N}\Big(2-\frac{1}{\ln 2}\Big)N\Big(\mu(n)-\frac{1}{N}\Big)^2\nonumber\\
&\quad+\sum_{\mu(n)>2/N}\Big(2-\frac{1}{\ln 2}\Big)\Big(\mu(n)-\frac{1}{N}\Big).\nonumber
\end{align}
Since $\sum\mu(n)=1$, the term \eqref{equation:linterm} vanishes.
Using the definitions of $f$ and $g$ we can write
\[
\log N-H(\mu)\ge\Big(2-\frac{1}{\ln 2}\Big)N\|f-\chi_N\|_2^2+\Big(2-\frac{1}{\ln 2}\Big)\|g\|_1.
\]
The claim now follows by noting that $2-1/\ln 2\ge1/2$.
\end{proof}

\subsection{}
\label{sc:L2}

The purpose of this section is to give a lower bound for the entropy of the convolution of
two functions whose $L^2$ distance from $\chi_N$ is small.

\begin{lem}\label{lm:L2-conv}
Let $N$ be a positive integer and
let $f,\wt f:\Z\to \R_{\ge0}$ be two functions concentrated on $[1,N]$
such that $\|f\|_\infty,\|\wt f\|_\infty \le 2/N$ and $\|f\|_1,\|\wt f\|_1\le 1$.
Suppose $2|N$ and let $M|N$.
Suppose further
\[
\|f-\chi_N\|_2^2\le \frac{1}{100N},\qquad\|\wt f-\chi_N\|_2^2\le \frac{1}{100N}.
\]
Write $\rho=(f*\wt f)|_{[N/2+1,3N/2]}$.

Then
\[
 \|\rho\|_1\log M - H(\rho;1|M)\le C\Big(N^2 \|f-\chiN\|_2^2\cdot \|\wt f-\chiN\|_2^2
+\frac{M\log M}{N}\Big).\quad\footnote{$C_1=1000$.\label{page:C1}}
\]
\end{lem}

First we give a lower bound for the entropy of a function whose $L^2$ distance from $\chi_M$ is small.
This is a partial converse to Lemma \ref{lm:decomp}.

\begin{lem}\label{lm:L2}
Let $M$ be an integer and let $\mu$ be a probability measure concentrated on $[1,M]\cap \Z$.
Then
\[
\log M-H(\mu)\le 2M\|\mu-\chiM\|_2^2.
\]
\end{lem}

\begin{proof}
We note the inequality
\[
x\log x\le \frac{1}{\ln 2}(x-1)+\frac{1}{\ln2}(x-1)^2.
\]
We sustitue $x=My$ and obtain
\[
y\log(My)\le \frac{1}{\ln 2}\Big(y-\frac{1}{M}\Big)+\frac{1}{\ln2}M\Big(y-\frac{1}{M}\Big)^2
\]
and
\[
\frac{\log M}{M}+y\log y\le \Big(\frac{1}{\ln 2}-\log M\Big)\Big(y-\frac{1}{M}\Big)
+\frac{1}{\ln2}M\Big(y-\frac{1}{M}\Big)^2.
\]

We substitute $y=\mu(n)$ and sum the resulting inequality for $n\in [1,M]$.
Since $\sum\mu(n)=1$, the linear term cancels and we obtain the inequality
claimed in the lemma using
$1/\ln 2<2$.
\end{proof}

In the next lemma, we show that the convolution of two functions with small $L^2$ distance from
$\chiN$ does not vary much on intervals shorter than the support.
This is essentially a consequence of the Cauchy-Schwartz inequality.

\begin{lem}\label{lm:diff}
Let  $N$ be a positive integer and let $f,\wt f: \Z\to \R_{\ge0}$ be
two functions concentrated on $[1,N]$ such that $\|f\|_\infty,\|\wt f\|_\infty<2/N$.
Then
\[
|f*\wt f(n)-f*\wt f(n+m)|\le \frac{3m}{N^2}+2\|f-\chiN\|_2\|\wt f-\chiN\|_2
\]
for any $n,m\in \Z$.
\end{lem}
\begin{proof}
Set $h=f-\chiN$ and $\wt h=\wt f-\chiN$.
Then
\begin{align}
f*\wt f(n)-&f*\wt f(n+m)\nonumber\\
=&\sum_{k\in\Z} (\chiN(k)\chiN(n-k)-\chiN(k)\chiN(n+m-k))\label{eq:sum1}\\
&+\sum_{k\in\Z}(h(k)\chiN(n-k)-h(k)\chiN(n+m-k))\label{eq:sum2}\\
&+\sum_{k\in\Z}(\chiN(n-k)\wt h(k)-\chiN(n+m-k)\wt h(k))\label{eq:sum3}\\
&+\sum_{k\in\Z}(h(k)\wt h(n-k)-h(k)\wt h(n+m-k)).\label{eq:sum4}
\end{align}

We first consider the contribution of the first three sums.
Without loss of generality we assume that $m>0$.
It is easy to see that $\chi_N(n-k)-\chi_N(n+m-k)\neq 0$ implies that
either $n-k\le 0$ and $n+m-k>0$ or $n-k\le N$ and $n+m-k>N$.
Rearranging these inequalities we obtain that $k$ satisfies
either $n\le k<n+m$ or $n-N\le k<n+m-N$.
If a term corresponding to $k$ in one of \eqref{eq:sum1}--\eqref{eq:sum3}
is non-zero, then $k$ must satisfy one of these inequalities and also $1\le k\le N$,
otherwise $\chi_N(k)$, $h(k)$ and $\wt h(k)$ are zero.

This shows that each of \eqref{eq:sum1}--\eqref{eq:sum3} have at most $m$ non-zero
terms.
Each term is bounded above by $N^{-2}$, as
$\|\chi_N\|_\infty,\| h\|_\infty,\|\wt h\|_\infty\le1/N$.
Hence the total contribution
of the first three sums is at most $3mN^{-2}$.

We estimate the fourth term using the Cauchy-Schwartz inequality:
\[
\Big|\sum_{k\in\Z}(h(k)\wt h(n-k)-h(k)\wt h(n+m-k))\Big|\le 2\|h\|_2\|\wt h\|_2.
\]
\end{proof}

We will convert the information obtained in the previous lemma to a lower bound on  entropy
using Lemma \ref{lm:L2}.
This requires to normalize the convolution to obtain probability measures on intervals of length $M$.
We need to show that the error is not magnified too much, hence we need to show that the convolution
has enough mass on each such interval.
This is done using the next lemma.

\begin{lem}\label{lm:lower}
Let $N$ be a positive even integer and let $f,\wt f:\Z\to\R_{\ge0}$ be two functions concentrated on $[1,N]$.
Suppose that
\[
\|f-\chiN\|_2\le \frac{1}{10 N^{1/2}}\quad\text{and}\quad
\|\wt f-\chiN\|_2\le \frac{1}{10 N^{1/2}}.
\]
Then
\[
f*\wt f(n)\ge \frac{1}{4N}
\]
for all $N/2+1\le n\le 3N/2$.
\end{lem}
\begin{proof}
Set $h=f-\chiN$ and $\wt h=\wt f-\chiN$.
We write
\begin{align}
f*\wt f(n)=&\sum_{k\in\Z} \chiN(k)\chiN(n-k)
+\sum_{k\in\Z}h(k)\chiN(n-k)\nonumber\\
&+\sum_{k\in\Z}\chiN(n-k)\wt h(k)\label{eq:secondline}
+\sum_{k\in\Z}h(k)\wt h(n-k).
\end{align}

We assume without loss of generality that $n\le N+1$.
Then
\[
\sum_{k\in\Z} \chi_N(k)\chi_N(n-k)=(n-1)/N^2.
\]
By the Cauchy-Schwartz inequality,
\[
\Big|\sum_{k\in\Z}h(k)\chi_N(n-k)\Big|=\sum_{k=1}^{n-1}\frac{1}{N}|h(k)|
\le\frac{(n-1)^{1/2}}{N}\|h\|_2\le \frac{(n-1)^{1/2}}{10 N^{3/2}}.
\]
One can derive a similar estimate for the first sum in \eqref{eq:secondline}.
Finally, we estimate the last sum by
\[
\Big|\sum_{k\in\Z}h(k)\wt h(n-k)\Big|\le \|h\|_2\|\wt h\|_2\le\frac{1}{100 N}.
\]

We note that 
\[
t-\frac{t^{1/2}}{5}-\frac{1}{100}>\frac{1}{4}
\]
for all $t\ge1/2$.
We apply this with $t=(n-1)/N$, which completes the proof.
\end{proof}

\begin{proof}[Proof of Lemma \ref{lm:L2-conv}]
We write 
\[
\rho_a:=\rho|_{[a+1,a+M]}
\]
for each $a\in\Z$.
We recall the formula \eqref{eq:ent-on-Z2}:
\[
H(\rho;1|M)=\frac{1}{M}\sum_{a\in\Z} H(\rho_a).
\]

Write $I=[N/2,3N/2-M]\cap\Z$ and $p(a):=\|\rho_a\|_1$ for $a\in \Z$.
Then  $p(a)\ge M/4N$ for $a\in I$ by Lemma \ref{lm:lower}.

The average of the values of $\rho_a(x)$ for $x\in[a+1,a+M]\cap\Z$ is $p(a)/M$, hence
\[
|\rho_a(x)-p(a)/M|<\frac{3M}{N^2}+2 \|f-\chiN\|_2\cdot \|\wt f-\chiN\|_2
\]
for all $x\in[a+1,a+M]$ by Lemma \ref{lm:diff}.

Combining the above two observations, and writing $\chi_{a+1,a+M}$ for the normalized
counting measure on $[a+1,a+M]\cap\Z$, we obtain for $a\in I$
\begin{align*}
\|p(a)^{-1}\rho_a-\chi_{a+1,a+M}\|_2^2<&\frac{16N^2}{M^2}\cdot
2M\Big(\frac{9M^2}{N^4}+4 \|f-\chiN\|_2^2\cdot \|\wt f-\chiN\|_2^2\Big)\\
<&400\Big(\frac{M}{N^2}+\frac{N^2}{M}\|f-\chiN\|_2^2\cdot \|\wt f-\chiN\|_2^2\Big).
\end{align*}
We apply Lemma \ref{lm:L2} for the probability measure $p(a)^{-1}\rho_a$ and obtain
\[
H(\rho_a)\ge p(a)\log M - 
800p(a)M\Big(\frac{M}{N^2}+\frac{N^2}{M}\|f-\chiN\|_2^2\cdot \|\wt f-\chiN\|_2^2\Big).
\]

For $a\notin I$,
we use the trivial estimate $H(\rho_a)\ge0$.
The number of $a\notin I$ such that $p(a)>0$ is at most $2M$.
Since $\|\rho\|_\infty\le \|f\|_\infty\le 2/N$, we have $p(a)\le 2M/N$.
Thus
\[
\sum_{a\notin I}p(a)\le \frac{4M^2}{N}.
\]

Note that
\[
\sum_{a\in\Z}p(a)=\sum_{a\in\Z}\sum_{i=1}^M\rho(a+i)=M\sum_{a\in\Z}\rho(a)=M.
\]
Then
\[
M\ge \sum_{a\in I}p(a)\ge M\|\rho\|_1-\frac{4M^2}{N}.
\]

We combine our estimates and obtain
\[
H(\rho;1|M)\ge\Big(\|\rho\|_1-\frac{4M}{N}\Big)\log M-
800\Big(\frac{M^2}{N^{2}}+N^2\|f-\chiN\|_2^2\cdot \|\wt f-\chiN\|_2^2\Big).
\]\footnote{The constant $C_1$ in the lemma needs to satisfy $C_1M\log M/N\ge 4M\log M/N+800 M^2/N^2$
so $C_1=1000$ works.}
\end{proof}

\subsection{}
\label{sc:L2L1}

The purpose of this section is to estimate the entropy of the convolution of a function
that is near constant in $L^2$-norm and one that has small $L^1$ norm.

\begin{lem}\label{lm:L2-L1-conv}
Let $N$ be a positive integer and let
$f,g:\Z\to \R_{\ge0 }$ be two functions concentrated on $[1,N]$ such that $1/2\le \|f\|_1\le1$ and 
$\|f\|_\infty\le2/N$.
Suppose that $2|N$ and let $M|N$.
Suppose further $\|g\|_1\le1$.
Let $\rho = (f*g)|_{[N/2+1,3N/2]}$.

Then
\[
\|\rho \|_1\log M-H(\rho;1|M)\le   8 N\|f-\chiN\|_2^2\|g\|_1+6\frac{M\log M}{N}.
\]
\end{lem}

We begin by recording a consequence of Lemma \ref{lm:L2}.

\begin{lem}\label{lm:L2-2}
Let $N$ be a positive integer and let 
$f:\Z\to \R_{\ge0}$ be a function concentrated on $[1,N]$ such that $1/2\le \|f\|_1\le 1$ and 
$\|f\|_\infty\le 2/N$.
Then for any $M|N$, we have
\[
\|f\|_1\log M-H(f;1|M)\le    8 N\|f-\chiN\|_2^2+2\frac{M}{N}.
\]
\end{lem}

\begin{proof}
Let $X$ be a random variable with law $\|f\|_1^{-1} f$.
We note
\[
\|\|f\|_1^{-1}f-\chi_N\|_2^2\le
\|\|f\|_1^{-1}f-\chi_N\|_2^2+\|(1-\|f\|_1^{-1})\chi_N\|_2^2=\|f\|_1^{-2}\|f-\chi_N\|_2^2,
\]
since $\|f\|_1^{-1}f-\chi_N$ is orthogonal to $\chi_N$.
We apply Lemma \ref{lm:L2} with $M=N$ for the law of $X$, and
obtain
\[
H(X)\ge \log N -  2N\|f\|_1^{-2}\|f-\chiN\|_2^2\ge \log N - 8 N\|f-\chiN\|_2^2.
\]

We can write using \eqref{eq:ent-on-Z1}
\begin{align*}
\|f\|_1^{-1}H(f;1|M)=&H(X;1|M)
=\frac{1}{M}\sum_{a=0}^{M-1} H\Big(X+a\Big|\Big\lfloor \frac{X+a}{M}\Big\rfloor\Big)\\
\ge& H(X)-\frac{1}{M}\sum_{a=0}^{M-1}H\Big(\Big\lfloor \frac{X+a}{M}\Big\rfloor\Big).
\end{align*}
The random variable $\lfloor (X+a)/M\rfloor$ may take at most $N/M+1$ different
values, hence $H(\lfloor (X+a)/M\rfloor)\le \log(N/M+1)$.

Combining our estimates, we obtain
\[
\|f\|_1^{-1}H(f;1|M)\ge \log N - 8 N\|f-\chiN\|_2^2 - \log(N/M+1).
\]

Since $\log(N/M+1)\le \log(N/M)+2 M/N$, we have
\[
\|f\|_1^{-1}H(f;1|M)\ge \log M - 8 N\|f-\chi_N\|_2^2 -2M/N.
\]
This proves the lemma.
\end{proof}

Lemma \ref{lm:L2-2} implies the conclusion of Lemma \ref{lm:L2-L1-conv} with $\rho$
replaced by $f$.
Since convolution may only increase entropy between scales of integral ratio,
it also implies the claim with $\rho$ replaced by $f*g$.
To conclude the proof of Lemma \ref{lm:L2-L1-conv},
it is left to consider the effect of taking restriction
to $[N/2+1,3N/2]$.

\begin{proof}[Proof of Lemma \ref{lm:L2-L1-conv}]
Write for $a\in \Z$
\[
\eta_a:=(f*g)|_{[a+1,a+M]}, \quad \rho_a:=\rho|_{[a+1,a+M]}.
\]
With this notation, we can write using \eqref{eq:ent-on-Z2}
\[
H(f*g;1|M)=\frac{1}{M}\sum_{a\in\Z}H(\eta_a), \quad
H(\rho;1|M)=\frac{1}{M}\sum_{a\in\Z}H(\rho_a).
\]

We observe that for $N/2\le a\le 3N/2-M$ we have $\eta_a=\rho_a$.
Using the trivial estimate $H(\eta_a)\le \|\eta_a\|_1\log M$ we can write
\[
H(\rho;1|M)\ge H(f*g;1|M) - \frac{\log M}{M}\sum_{a\in\Z\backslash[N/2,3N/2-M]} \|\eta_a\|_1.
\]

We use now that convolving $f$ by $g$ may only increase its entropy (up to normalization),
and then apply Lemma \ref{lm:L2-2}:
\[
H(f*g;1|M)\ge \|g\|_1H(f;1|M)\ge \|f*g\|_1\log M - 8 N\|f-\chiN\|_2^2\|g\|_1-2\frac{M}{N}.
\]

We note that
\begin{align*}
\frac{1}{M} \sum_{a\in\Z\backslash[N/2,3N/2-M]} \|\eta_a\|_1
\le& \sum_{n\in \Z\backslash [N/2+M,3N/2-M+1]}f*g(n)\\
\le&\|f*g\|_1-\|\rho\|_1+(2M-2)\|f*g\|_\infty.
\end{align*}

We add that $\|f*g\|_\infty\le 2/N$ and combine the above estimates:
\begin{align*}
H(\rho;1|M)\ge& \|f*g\|_1\log M - 8 N\|f-\chiN\|_2^2\|g\|_1-2\frac{M}{N}\\
&-\Big(\|f*g\|_1-\|\rho\|_1+\frac{4M}{N}\Big)\log M\\
\ge&\|\rho\|_1\log M - 8 N\|f-\chiN\|_2^2\|g\|_1-6 \frac{M\log M}{N}.
\end{align*}
\end{proof}

\subsection{Proof of Proposition \ref{pr:restricted-conv}}
\label{sc:restricted-conv}

We first consider the case, when $\log N-H(\mu)<1/200$ and
$\log N-H(\wt\mu)<1/200$.
We apply Lemma \ref{lm:decomp} to the measures $\mu$ and $\wt\mu$ and write
$\mu=f+g$ and $\wt\mu=\wt f+\wt g$ such that
\begin{align}
\label{eq:f-conditions}
\|f-\chiN\|_2^2&\le2\frac{\log N-H(\mu)}{N}, \qquad \|f\|_\infty\le\frac{2}{N},\\
\label{eq:g-conditions}
\|g\|_1&\le2 (\log N-H(\mu))<\frac12,\\
\label{eq:ftilde-conditions}
\|\wt f-\chiN\|_2^2&\le2\frac{\log N-H(\wt \mu)}{N}, \qquad \|\wt f\|_\infty\le\frac{2}{N},\\
\label{eq:gtilde-conditions}
\|\wt g\|_1&\le2 (\log N-H(\wt\mu))<\frac12.
\end{align}
Since $\mu=f+g$ and both $f$ and $g$ are non-negative, we have $\|f\|_1+\|g\|_1=1$, hence $\|f\|_1,\|g\|_1\le1$
and we have $\|f\|_1> 1/2$ by \eqref{eq:g-conditions}.

We put $\rho_1=(f*\wt f)|_{[N/2+1,3N/2]}$, $\rho_2=(f*\wt g)|_{[N/2+1,3N/2]}$
and $\rho_3=(g*\wt f)|_{[N/2+1,3N/2]}$.
The conditions of Lemmata \ref{lm:L2-conv} and \ref{lm:L2-L1-conv} hold,
hence we can write
\begin{align*}
H(\rho_i;1|M)\ge&\|\rho_i\|_1\log M -  C(\log N-H(\mu))(\log N -H(\wt\mu))\quad\footnotemark\\ 
&-C \frac{M\log M}{N}\quad\footnotemark
\end{align*}
\addtocounter{footnote}{-1}
\footnotetext{$C_2=4000$. If $i=1$, we apply Lemma \ref{lm:L2-conv} and we need $C_2\ge C_1\cdot 2^2=4000$.
If $i=2,3$, we apply Lemma \ref{lm:L2-L1-conv} and we need $C_2\ge 8\cdot2\cdot2$.
Here $C_1$ is the constant from Lemma \ref{lm:L2-conv} on page \pageref{page:C1}.}
\stepcounter{footnote}
\footnotetext{$C_3=1000$, which is the maximum of $C_1$ and the constant $6$ from Lemma \ref{lm:L2-L1-conv}.}
for $i=1,2,3$.

We note that
\[
\|\sigma\|_1-(\|\rho_1\|_1+\|\rho_2\|_1+\|\rho_3\|_1)\le\|g*\wt g\|_1
\le4(\log N-H(\mu))(\log N -H(\wt\mu)).
\]
This proves the proposition, since
\[
\s=\rho_1+\rho_2+\rho_3+(g*\wt g)|_{[N/2+1,3N/2]},
\]
and entropy between scales of integral ratio is superadditive by Lemma~\ref{lm:conditional-super-additive}.
\footnote{The constant $C_4$ in the proposition needs to absorb $3\cdot C_2$ coming from the lower bound on $H(\rho_i;1|M)$
for $i=1,2,3$ plus $4$, which comes from $(\|\s\|_1-(\|\rho_1\|_1+\|\rho_2\|_1+\|\rho_3\|_1))\log M$.
This holds for $C_4=4\cdot 10^4$.
The constant $C_5=3000$, needs to satisfy $C_5\ge 3\cdot C_3$.}

Next, we consider the case, when $\log N-H(\mu)<1/200$ and
$\log N-H(\wt\mu)\ge1/200$.
We apply Lemma \ref{lm:decomp} to the measure $\mu$ and write
$\mu=f+g$ with functions $f$ and $g$ that satisfy \eqref{eq:f-conditions}
and $\eqref{eq:g-conditions}$.

We put $\rho=(f*\wt\mu)|_{[N/2+1,3N/2]}$ and apply Lemma \ref{lm:L2-L1-conv}
to get
\[
H(\rho;1|M)\ge \|\rho\|_1\log M-16(\log N-H(\mu))
-6\frac{M\log M}{N}.
\]

We note that
\[
\|\sigma\|_1-\|\rho\|_1\le\|g\|_1\le 2 (\log N-H(\mu)).
\]

We consider the identity $\sigma=\rho+(g*\wt\mu)_{[N/2+1,3N/2]}$
and conclude the claim by superadditivity of entropy.
\footnote{The constant $C_4=4\cdot10^4$ in the proposition needs to absorb $16/(\log N-H(\wt\mu))\le16\cdot 200=3200$ plus $2/(\log N-H(\wt\mu))\le 400$.
For $C_5$ we only need $C_5\ge 6$.}

The case $\log N-H(\mu)\ge1/200$ and
$\log N-H(\wt\mu)<1/200$ is the same as the previous.
If $\log N-H(\mu)\ge1/200$ and
$\log N-H(\wt\mu)\ge1/200$, then the proposition is vacuous.
\footnote{It is vacuous, because $C_4\ge200^2$.}

\subsection{Proof of Proposition \ref{pr:high-ent-Z}}
\label{sc:proof-high-ent-Z}

For each $a,b\in \Z$ write $f_a=\nu|_{[a+1,a+N]}$, $\wt f_b=\wt\nu|_{[b+1,b+N]}$.
By \eqref{eq:ent-on-Z2} we have
\begin{equation}\label{eq:ent-nu}
H(\nu;1|N)=\frac{1}{N}\sum_{a\in \Z}H(f_a),\quad
H(\wt\nu;1|N)=\frac{1}{N}\sum_{b\in \Z}H(\wt f_b).
\end{equation}

We put
\[
\s_{a,b}=f_a*\wt f_b|_{[a+b+N/2+1,a+b+3N/2]}.
\]
We will show below that
\begin{equation}\label{eq:3-4}
\frac{3}{4}N^2\nu*\wt\nu=\sum_{a,b\in \Z}\s_{a,b}.
\end{equation}

Taking \eqref{eq:3-4} for granted, we complete the proof.
We apply Proposition \ref{pr:restricted-conv} for the probability measures
$\mu=\|f_a\|^{-1}f_a*\d_{-a}$ and $\wt \mu=\|\wt f_b\|^{-1}\wt f_b*\d_{-b}$
and obtain
\begin{align*}
H(\s_{a,b};1|M)\ge& \|\s_{a,b}\|_1 \log M\\
&-C\log M (\|f_a\|_1\log N-H(f_a))(\|\wt f_b\|_1\log N-H(\wt f_b))\quad\footnotemark\\
&-C\|f_a\|_1\|\wt f_b\|_1\frac{M\log M}{N}.\quad\footnotemark
\end{align*}
\addtocounter{footnote}{-1}
\footnotetext{$C_4=4\cdot 10^4$, (see page \pageref{page:C4}).}
\stepcounter{footnote}
\footnotetext{$C_5=3000$, (see page \pageref{page:C5}).}
\addtocounter{footnote}{-2}

We sum the above inequality for $a,b\in\Z$.
We use \eqref{eq:3-4} and superadditivity of entropy to conclude
\begin{align*}
\frac{3}{4}N^2&H(\nu*\wt \nu;1|M)\ge\sum_{a,b\in\Z}\|\s_{a,b}\|_1\log M\\
&-C\log M \sum_{a,b\in\Z}(\|f_a\|_1\log N-H(f_a))(\|\wt f_b\|_1\log N-H(\wt f_b))\quad\footnotemark\\
&-C\sum_{a,b\in\Z}\|f_a\|_1\|\wt f_b\|_1\frac{M\log M}{N}.\quad\footnotemark
\end{align*}

\addtocounter{footnote}{-2}

We use  \eqref{eq:ent-nu}, $\sum\|f_a\|_1=N$, $\sum\|\wt f_b\|_1=N$ and
$\sum\|\s_{a,b}\|_1=3N^2/4$.
The latter is a consequence of \eqref{eq:3-4}.
We obtain
\begin{align*}
\frac{3}{4}N^2&H(\nu*\wt \nu;1|M)\ge\frac{3}{4}N^2\log M\\
&-C\log M (N\log N-NH(\nu;1|N))(N\log N-NH(\wt \nu;1|N))\quad\footnotemark\\
&-CN^2\frac{M\log M}{N},\quad\footnotemark
\end{align*}
which proves the claim upon dividing both sides by $3N^2/4$.
\footnote{We need the constants $C_6$ and $C_7$ in the proposition (see page \pageref{page:C6})
to satisfy $C_6=6\cdot 10^4\ge (4/3)C_4$ and $C_7=4000\ge(4/3)C_5$.}

It is left to prove \eqref{eq:3-4}.
We note that both sides of \eqref{eq:3-4} are linear in both $\nu$ and $\wt \nu$, therefore
it is enough to prove it for $\nu=\delta_x$ and $\wt\nu=\delta_y$ for every $x,y\in\Z$.
In this case, $\nu*\wt\nu=\d_{x+y}$.
In addition, we have
$\s_{a,b}= \d_{x+y}$ if the three conditions
\begin{align*}
x\in&[a+1,a+N],\\
y\in&[b+1,b+N],\\
x+y\in&[a+b+N/2+1,a+b+3N/2]
\end{align*}
hold, and $\s_{a,b}=0$ otherwise.
It is easy to see that for any $x,y$, there are $3N^2/4$ choices of pairs $(a,b)\in\Z^2$ such that
$\s_{a,b}=\d_{x,y}$ and this proves \eqref{eq:3-4}.

\subsection{}
\label{sc:proof-high-ent-th}
The purpose of this section is to explain the reduction of Theorem \ref{th:high-entropy-convolution}
to Proposition \ref{pr:high-ent-Z}.

\begin{proof}[Proof of Theorem \ref{th:high-entropy-convolution}]
We put $K=\lfloor 3\log\a^{-1}\rfloor$.
By rescaling $\mu$ and $\wt \mu$ if necessary, we may assume that $r=2^K$.
We define the probability measures $\nu$ and $\wt\nu$ on $\Z$ by
\[
\nu(n)=\mu([n,n+1)),\quad \wt\nu(n)=\wt\mu([n,n+1)).
\]

By \eqref{equation:interpret2} applied twice, $H(\mu;1|2^{2K+1})$ is the average of 
\[
H(\{\mu([n+t,n+1+t))\}_{n\in\Z};1|2^{2K+1})
\]
with $t$ running over $[0,1)$,
and a similar relation holds for $\wt\mu$.
Therefore, by replacing $\mu$ or $\wt\mu$ or both by suitable translates,
we may assume, that
\begin{align*}
H(\nu;1|2^{2K+1})\ge& H(\mu;1|2^{2K+1})\\
H(\wt\nu;1|2^{2K+1})\ge& H(\wt\mu;1|2^{2K+1}).
\end{align*}
Using the hypothesis $H(\mu;r|2r)\ge 1-\a$ for $r=2^a$ for $a=0,\ldots, 2K$
(and similar inequalities for $\wt \mu$) these yield
\begin{align}
H(\nu;1|2^{2K+1})\ge& (2K+1)(1-\a),\label{eq:alpha1}\\
H(\wt\nu;1|2^{2K+1})\ge& (2K+1)(1-\a).\label{eq:alpha2}
\end{align}

Lemma \ref{lm:perturb} gives
\begin{equation}\label{eq:large-conv1}
|H(\mu*\wt\mu;2^K|2^{K+1})-H(\nu*\wt\nu;2^K|2^{K+1})|\le 4 \frac{K}{2^{K}}.
\end{equation}

We use Proposition \ref{pr:high-ent-Z} with $M=2^{K+1}$ and $N=2^{2K+1}$ and get
\begin{align*}
K+&1-H(\nu*\wt\nu;1|2^{K+1})\\
\le& CK(2K+1-H(\nu;1|2^{2K+1}))\times(2K+1-H(\wt\nu;1|2^{2K+1}))\quad\footnotemark\\
&+C\frac{K}{2^K}.\quad\footnotemark
\end{align*}
\addtocounter{footnote}{-1}
\footnotetext{$C_8=1.2\cdot 10^5$. Indeed, we need $KC_8\ge (K+1) C_6$. For $C_6$ see page \pageref{page:C6}.}
\stepcounter{footnote}
\footnotetext{$C_{9}=8000$. Indeed, we need $KC_{9}\ge (K+1)C_7$. For $C_7$ see page \pageref{page:C7}.}
We combine this with
\begin{align*}
H(\nu*\wt\nu;2^K|2^{K+1})=&H(\nu*\wt\nu;1|2^{K+1})-H(\nu*\wt\nu;1|2^K)\\
\ge& H(\nu*\wt\nu;1|2^{K+1})-K
\end{align*}
and with \eqref{eq:alpha1}--\eqref{eq:alpha2} and write
\begin{align}\label{eq:large-conv2}
1-H(\nu*\wt\nu;2^K|2^{K+1})\le K+1-H(\nu*\wt\nu;1|2^{K+1})
\le CK^3\a^2\quad\footnotemark\\
+C\frac{K}{2^{K}}.\quad\footnotemark\nonumber
\end{align}
\addtocounter{footnote}{-1}
\footnotetext{$C_{10}=1.1\cdot 10^6$. Indeed, we need $K^2C_{10}\ge (2K+1)^2 C_8$.}
\stepcounter{footnote}
\footnotetext{$C_{9}=8000$.}
By the choice of $K=\lfloor 3\log\a^{-1}\rfloor$ and $\a<1/2$,
we have $2^{-K}<\a^2$ and the claim follows
if we combine \eqref{eq:large-conv1} and \eqref{eq:large-conv2}. 
\footnote{For the constant $C$ in the theorem, we need $C\ge27C_{10}+3(C_{9}+4)$, which holds if we set $C=10^{8}$.
Here the factor $27$ comes from the estimate $K^3\le27(\log \a^{-1})^3$ and the factor $3$ comes from $K\le3(\log \a^{-1})^3$.}
\end{proof}

\section{Entropy of convolutions in the low entropy regime}
\label{sc:low-entropy-convolution}

The purpose of this section is to prove Theorem \ref{th:low-entropy-convolution}, which we
restate.

\begin{thm*}
For every $0<\a<1/2$, there is a number $c>0$ such that the following holds.
Let $\mu,\nu$ be two compactly supported probability measures on $\R$.
Let  $\s_2<\s_1<0$ and $0< \b\le1/2$ be real numbers.
Suppose that 
\begin{equation}\label{eq:few-uniform-scales}
\cN_1\{ \s\in[\s_2,\s_1]:H(\mu;2^\s|2^{\s+1})>1-\a\}<c\b(\s_1-\s_2).
\end{equation}
Suppose further that
\[
H(\nu; 2^{\s_2}|2^{\s_1})>\b(\s_1-\s_2).
\]

Then
\[
H(\mu*\nu; 2^{\s_2}|2^{\s_1})> H(\mu;2^{\s_2}|2^{\s_1})+c\b(\log\b^{-1})^{-1}(\s_1-\s_2)-3.
\]
\end{thm*}

Our proof of Theorem \ref{th:low-entropy-convolution} is motivated by some
ideas of Bourgain in his second proof of the discretized ring conjecture
\cite{Bou-discretized2}.
The proof relies on the following
two propositions.
We call a measure a Bernoulli measure, if it is supported on two points,
which have equal weight (not necessarily $1/2$, unless it is a probability measure).

\begin{prp}\label{pr:conv-by-Bernoulli}
Let $\mu$  be a compactly supported probability measure on $\R$ and
let $t,r_1,r_2>0$ be numbers.
Let $\nu$ be a Bernoulli probability measure supported on two points at distance $t$.
Then
\[
H(\mu*\nu;r_2|r_1)\ge H(\mu;r_2|r_1)+\frac{1}{3}(1-H(\mu;t|2t)),
\]
provided
\[
r_2\le t(1-H(\mu;t|2t))/10\quad \text{and}\quad r_1\ge 144t(1-H(\mu;t|2t))^{-2}.
\]
\end{prp}

\begin{prp}\label{pr:bernoulli-decompose}
Let $\mu$ be a finitely supported probability measure on $\R$
and let $r>0$ be a number.
Suppose that 
\[
H\Big(\mu;\frac{r}{2}\Big|r\Big)\le 1.5\cdot H(\mu;r|2r).
\]

Then
\[
\mu=\nu+\eta_1+\ldots+\eta_N,
\]
where $N$ is an integer, $\nu$ is a non-negative measure,
\[
\|\eta_1\|+\ldots+\|\eta_N\|\ge \frac{1}{128}\cdot \frac{H(\mu;r|2r)}{\log(H(\mu;r|2r)^{-1})+1}
\]
and $\eta_i$ are Bernoulli measures supported on pairs of points at distances
between $2r$ and $r/2$.
\end{prp}

A decomposition similar to the one in Proposition \ref{pr:bernoulli-decompose}
appears in \cite{LV-entropy-sum-product}, however, they have different quantitative aspects,
hence they require different proofs.
The main difference between the setup in Theorem \ref{th:low-entropy-convolution}
and \cite{LV-entropy-sum-product} is that assumption \eqref{eq:few-uniform-scales}
is absent in the latter, and this makes a drastic difference in the quantitative features
of the conclusion, and also very different arguments are required.

The proofs of Propositions \ref{pr:conv-by-Bernoulli} and \ref{pr:bernoulli-decompose}
will be given in Sections \ref{sc:conv-by-Bernoulli} and \ref{sc:bernoulli-decompose}, respectively.

In the proof of Theorem \ref{th:low-entropy-convolution}, we will use Proposition \ref{pr:bernoulli-decompose}
for the measure $\nu$ to write it as a combination of Bernoulli measures.
We will see that the distance between the points, where the Bernoulli measures are supported can
be choosen to fall in many different scale ranges.

We will then apply Proposition \ref{pr:conv-by-Bernoulli} to show that we gain a small amount
of entropy on each such scale range.
The details of this argument are given in Section \ref{sc:low-entropy-proof}.

\subsection{}
\label{sc:conv-by-Bernoulli}

The purpose of this section is the proof of Proposition \ref{pr:conv-by-Bernoulli}.
We begin with a simple observation about how entropy increases
if we convolve a measure by a Bernoulli measure supported at points of distance matching the scale.

\begin{lem}\label{lm:conv-by-Bernoulli}
Let $\mu$  be a compactly supported probability measure on $\R$ and
let $t>0$ be number.
Let $\nu$ be a Bernoulli probability measure supported on two points at distance $t$.
Then
\[
H(\mu*\nu;t)=H(\mu;t)+(1-H(\mu;t|2t)).
\]
\end{lem}
\begin{proof}
We assume as we may that $\nu=(\d_0+\d_t)/2$.
We use the alternative definition for the entropies given in Lemma \ref{lm:second-def}.
Denoting by $\chi_s$ the normalized Lebesgue measure on $[0,s]$, we can write
\begin{align*}
H(\mu*\nu;t)=&H(\mu*\nu*\chi_t)-H(\chi_t)\\
=&H(\mu*\chi_{2t})-H(\chi_{2t})+(H(\chi_{2t})-H(\chi_t))\\
=&H(\mu;2t)+1\\
=&H(\mu;t)-H(\mu;t|2t)+1.
\end{align*}
\end{proof}

In order to prove Proposition \ref{pr:conv-by-Bernoulli},
we need to show that the entropy increase obtained in Lemma \ref{lm:conv-by-Bernoulli}
is captured by a suitably large
scale range around the distance between the points in the support of $\nu$.
We proceed with the details of this.

\begin{proof}[Proof of Proposition \ref{pr:conv-by-Bernoulli}]
Applying Lemma \ref{lm:conv-by-Bernoulli}, we obtain
\[
H(\mu*\nu;t)=H(\mu;t)+(1-H(\mu;t|2t)).
\]

By Lemma \ref{lm:perturb} we have
\[
|H(\mu*\nu;r_1)-H(\mu;r_1)|\le (2 t/r_1)\log(r_1/t)\le 4 (t/r_1)^{1/2}\le (1-H(\mu;t|2t))/3,
\]
where we used the inequality $\log(x)\le 2x^{1/2}$,
and then the assumption
\[
r_1/t\ge144(1-H(\mu;t|2t))^{-2}.
\]
Hence
\begin{equation}\label{eq:cbB1}
H(\mu*\nu;t|r_1)\ge H(\mu;t|r_1)+\frac23(1-H(\mu;t|2t)).
\end{equation}

Using  Lemma \ref{lm:fractional-scale}, we write
\[
H(\mu*\nu;r_2|t)\ge H(\mu;r_2|t) - \frac{2}{(\ln 2)(t/r_2-1)}.
\]
By assumption,
\[
t/r_2\ge 10(1-H(\mu;t|2t))^{-1}\ge \frac{6}{\ln 2} (1-H(\mu;t|2t))^{-1}+1,
\]
where we also used $6/\ln(2)<9$.
This yields
\[
\frac{2}{(\ln 2)(t/r_2-1)}\le \frac13(1-H(\mu;t|2t)).
\]
Thus
\[
H(\mu*\nu;r_2|t)\ge H(\mu;r_2|t)- \frac13(1-H(\mu;t|2t)).
\]
This combined with \eqref{eq:cbB1} proves the claim.
\end{proof}

\subsection{}
\label{sc:bernoulli-decompose}

The purpose of this section is the proof of Proposition \ref{pr:bernoulli-decompose}.
Our first aim is the next lemma.

\begin{lem}\label{lm:bernoulli-decompose}
Let $\mu$ be a probability measure on $\R$ and let $0<r_0,r_1$ be numbers with $4r_0\le r_1$.
Let $I_1,\ldots I_n\subset\R$ be
disjoint intervals of length at most $r_0$ such that every two of them have a gap of at least $r_1$ between them.
Suppose $\supp\mu\subset I_1\cup\ldots\cup I_n$.
Then for all $2r_0\le r\le r_1/2$, we have
\[
H\Big(\mu;\frac{r}{2}\Big|r\Big)=2H(\mu;r|2r).
\]
\end{lem}

This lemma shows that the inequality
\[
H\Big(\mu;\frac{r}{2}\Big|r\Big)<2H(\mu;r|2r)
\]
implies that the support of $\mu$ contains at least one pair of points of distance
comparable to $r$.
We prove this lemma in Section \ref{sc:B-decompose1}.
In Section \ref{sc:B-decompose2},
we estimate the effect of restricting a measure to a subset of its support
on its entropy.
In Section \ref{sc:B-decompose3}, we use these estimates to understand
the effect of removing all pairs of points at distance comparable to $r$
from the support of $\mu$.
We combine this with Lemma \ref{lm:bernoulli-decompose} to conclude the proof
of Proposition \ref{pr:bernoulli-decompose}.

\subsubsection{}\label{sc:B-decompose1}
We introduce some notation.
Let $\mu$ be a probability measure on $\R$.
We write
\[
H_\pm(\mu):=-\mu(-\infty,0)\log\mu(-\infty,0)-\mu[0,\infty)\log\mu[0,\infty).
\]

We begin with the special case of Lemma \ref{lm:bernoulli-decompose}
in which the measure is supported on a small interval.

\begin{lem}\label{lm:small-interval}
Let $\mu$ be probability measure whose support is contained in an interval of length $r_0$.
Then for all $r\ge r_0$ we have
\begin{equation}\label{eq:small-interval}
H(\mu;r|2r)=\frac{1}{2r}\int_{-\infty}^{\infty} H_\pm(\mu*\d_x) dx.
\end{equation}
\end{lem}

Observe that the lemma indeed implies
\[
H\Big(\mu;\frac r2\Big|r\Big)=2H(\mu;r|2r)
\]
if $r\ge 2r_0$.

\begin{proof}
Since both sides of \eqref{eq:small-interval} are continuous in $r$, we may assume $r>r_0$.
Let $X$ be a random variable with law $\mu$.
We have
\[
H(\mu;r)=\frac{1}{r}\int_{0}^{r}H(\lfloor r^{-1} (X+t)\rfloor)dt.
\]

We assume without loss of generality that $\mu$ is concentrated on $[-r_0-\e,0-\e]$ for some $\e<r-r_0$.
Observe that the value of $\lfloor r^{-1} (X+t)\rfloor$ for $t\in[0,r]$ depends only the sign of $ r^{-1} (X+t)$.
Thus
\[
H(\mu;r)=\frac{1}{r}\int_{0}^{r}H_\pm( r^{-1} (X+t))dt
=\frac{1}{r}\int_{0}^{r}H_\pm(\mu*\d_t)dt
=\frac1r\int_{-\infty}^\infty H_\pm(\mu*\d_t)dt.
\]
The last equality follows from the fact that the integrand is $0$ for $t\notin [0,r]$.
We take the difference of this with itself with $2r$ substituted in place of $r$ and obtain the claim.
\end{proof}

We continue with a lemma which allows to reduce the general
case of Lemma \ref{lm:bernoulli-decompose}
to the special case considered in the previous lemma.

\begin{lem}\label{lm:decompose-intervals}
Let $\mu$ be a probability measure on $\R$ and let $r_1>0$ be a number.
Let $I_1,\ldots I_n\subset\R$ be
disjoint intervals such that every two of them have a gap of at least $r_1$ between them.
Suppose $\supp\mu\subset I_1\cup\ldots\cup I_n$.
Then for all $r\le r_1$, we have
\begin{equation}\label{eq:decompose-intervals}
H(\mu;r)=\sum_{j=1}^n H((\mu|_{I_j});r).
\end{equation}
\end{lem}

\begin{proof}
Since both sides of \eqref{eq:decompose-intervals} are continuous in $r$, we may assume
$r<r_1$.
We note that
\[
\mu*\chi_r=(\mu|_{I_1})*\chi_r+\ldots+(\mu|_{I_n})*\chi_r,
\]
and the measures on the right hand side have disjoint support.
Using Lemma \ref{lm:second-def}, we can write
\begin{align*}
H(\mu;r)=&H(\mu*\chi_r)-H(\chi_r)
=\sum_{j=1}^n (H((\mu|_{I_j})*\chi_r)-\mu(I_j)H(\chi_r))\\
=&\sum_{j=1}^n H((\mu|_{I_j});r).
\end{align*}
\end{proof}

Finally, we can prove the general version of our claim.

\begin{proof}[Proof of Lemma \ref{lm:bernoulli-decompose}]
We can decompose $\mu$ as the sum of its restrictions to each $I_j$.
By Lemma \ref{lm:decompose-intervals} applied with $r/2$, $r$ and $2r$
in place of $r$, it is enough to show the claim for each of these restrictions.
Hence we can assume without loss of generality that $\supp \mu$ is contained in
an interval of length $r_0$.
The claim now follows immediately from the formula in Lemma \ref{lm:small-interval}.
\end{proof}

\subsubsection{}\label{sc:B-decompose2}
We need a further technical lemma that allows us to compare the entropy of a measure with the
entropy of a term appearing in a decomposition.

\begin{lem}\label{lm:entropy-component}
Let $\mu$ be a probability measure.
Suppose that $\mu=\nu+\eta$ for two non-negative measures $\nu$ and $\eta$.
Suppose further that $\|\eta\|\le1/2$.
Then
\[
H(\nu;r|2r)\le H(\mu;r|2r)\le H(\nu;r|2r)+3\|\eta\|\log\|\eta\|^{-1}
\]
for any $r>0$.
\end{lem}

We remark that this lemma is closely related to Fano's inequality (see \cite{cover-thomas}*{Proposition 2.10.1})
and the inequality on the left hand side also follows from
Lemma \ref{lm:conditional-super-additive}.

\begin{proof}
Owing to \eqref{eq:scale-scaling}, we can assume without loss of generality that $r=1$.
Let $X$ be a random variable with law $\mu$ and let $Z$ be $\{0,1\}$ valued random variable
such that $\P(Z=0)=\|\nu\|$ and the distribution of $X$ conditioned on the event $Z=0$
is $\|\nu\|^{-1}\nu$, and the distribution of $X$ conditioned on $Z=1$ is $\|\eta\|^{-1}\eta$.

By \eqref{equation:interpret2}, we have
\begin{align*}
H(\mu;1|2)=&\int_0^1 H\big(\lfloor X+2t\rfloor\big|\lfloor X/2+t\rfloor\big) dt,\\
H(\nu;1|2)=&\P(Z=0)\int_0^1 H\big(\lfloor X+2t\rfloor\big|\lfloor X/2+t\rfloor,Z=0\big) dt,\\
H(\eta;1|2)=&\P(Z=1)\int_0^1 H\big(\lfloor X+2t\rfloor\big|\lfloor X/2+t\rfloor,Z=1\big) dt.
\end{align*}
Here we use the following convention for conditioning on events.
If $Y_1, Y_2$ are random variables and $E$ is an event, then to calculate $H(Y_1|Y_2,E)$,
we restrict the probability space to the event $E$ (and normalize the measure) and calculate
the conditional entropy of the restriction of the random variables to this new probability space.

This means that for each $t$, we have
\begin{align*}
H\big(\lfloor X+2t\rfloor\big|\lfloor X/2+t\rfloor,Z\big)=&
\P(Z=0) H\big(\lfloor X+2t\rfloor\big|\lfloor X/2+t\rfloor,Z=0\big)\\
&+ \P(Z=1) H\big(\lfloor X+2t\rfloor\big|\lfloor X/2+t\rfloor,Z=1\big)
\end{align*}
We combine this with the estimates
\begin{align*}
H\big(\lfloor X+2t\rfloor\big|\lfloor X/2+t\rfloor\big)\ge&
H\big(\lfloor X+2t\rfloor\big|\lfloor X/2+t\rfloor,Z\big)\\
\ge& H\big(\lfloor X+2t\rfloor\big|\lfloor X/2+t\rfloor\big)-H(Z).
\end{align*}
and integrate it for $t$.
We find
\[
H(\mu;1|2)\ge H(\nu;1|2)+H(\eta;1|2)\ge H(\mu;1|2)- H(Z).
\]

We note that $H(\eta;1|2)\le \|\eta\|$ and
\[
H(Z)=-\|\eta\|\log\|\eta\|-(1-\|\eta\|)\log(1-\|\eta\|)
\le 2\|\eta\|\log\|\eta\|^{-1},
\]
which proves the claim.
\end{proof}

\subsubsection{Proof of Proposition \ref{pr:bernoulli-decompose}}
\label{sc:B-decompose3}
Let
\[
\mu=\nu+\eta_1+\ldots+\eta_N
\]
be a decomposition such that $\nu$ is a non-negative measure, each $\eta_i$
is a Bernoulli measure supported on a pair of points at distance between $r/2$ and $2r$
and $\|\nu\|$ is minimal among all such decompositions.
Recall that $\mu$ is assumed to be finitely supported, hence the minimum exists.

Then there are no two points in the support of $\nu$ at  distance between $r/2$ and $2r$.
It is easy to see that the support of $\nu$ can be covered by intervals of length less than
$r/2$ that are of distance more than $2r$.
Hence Lemma \ref{lm:bernoulli-decompose} applies with $r_0=r/2$ and $r_1=2r$ and we have
\[
H\Big(\nu;\frac{r}{2}\Big|r\Big)=2H(\nu;r|2r).
\]

Write $\d=1-\|\nu\|$.
If $\|\d\|\ge 1/2$, the claim of the proposition holds trivially, so we assume that this
is not the case.
Then, by Lemma \ref{lm:entropy-component},
\[
H(\mu;r|2r)\le H(\nu;r|2r)+3\d\log\d^{-1}.
\]

Using Lemma \ref{lm:entropy-component} again, we can write
\[
H\Big(\mu;\frac r2\Big|r\Big)\ge H\Big(\nu;\frac r2\Big|r\Big)=2H(\nu;r|2r)
\ge 2 H(\mu;r|2r)-6\d\log\d^{-1}.
\]
Combining this with our assumption, we get
\[
1.5 H(\mu;r|2r)\ge 2 H(\mu;r|2r)-6\d\log\d^{-1},
\]
hence
\[
\d\log\d^{-1}\ge H(\mu;r|2r)/12.
\]

Now, we suppose to the contrary that the claim is false, that is
\[
\d\le\frac{1}{128}\cdot \frac{ H(\mu;r|2r)}{\log( H(\mu;r|2r)^{-1})+1}.
\]
Writing $h=H(\mu;r|2r)$, and using that $\d\mapsto \d\log \d$ is monotone, we get
\[
\d\log \d^{-1}\le \frac{1}{128}\cdot \frac{ h}{\log(h^{-1})+1}
\cdot (7+\log (h^{-1})+\log(\log(h^{-1})+1))
\le \frac{7h}{128},
\]
a contradiction, since $7/128<1/12$.
\subsection{}
\label{sc:low-entropy-proof}

The purpose of this section is the proof of Theorem \ref{th:low-entropy-convolution}.
We begin with a technical lemma that locates a large number of scales, where both
Propositions \ref{pr:conv-by-Bernoulli} and \ref{pr:bernoulli-decompose} can be applied.

\begin{lem}\label{lm:low-entropy-proof}
For every $0<\a<1/2$, there is a number $c'>0$ such that the following hold.
Let $\mu,\nu,\s_1,\s_2$ and $\b$ be as in Theorem \ref{th:low-entropy-convolution}
and assume that the hypotheses of that theorem hold.
Let $K=\lceil\log(144/\a^2)\rceil+2$.

Then there is a $2K$-separated set $B\subset \Z\cap [\s_2+K,\s_1-K]$ such that
 each $n\in B$ satisfies
\begin{align*}
1.5\cdot H(\nu;2^n|2^{n+1})&> H(\nu; 2^{n-1}|2^n),\\
 H(\nu;2^n|2^{n+1})&\ge \b/12,\\
H(\mu; t|2t)&\le 1-\a \text{ for all $t$ with $2^{n-1}\le t\le 2^{n+1}$.}
\end{align*}
Furthermore
\begin{equation}\label{eq:B-large}
\sum_{n\in B}H(\nu;2^n|2^{n+1})\ge c'\b(\s_1-\s_2)-1.
\quad\footnote{$c'=c_{11}=(1000\log \a^{-1})^{-1}$.}
\end{equation}
\end{lem}

The lemma would hold for any $K$ with the constant $c'$ depending on $K$.
Our choice of the value of $K$ will become relevant later.

\begin{proof}
Write
\[
a=\lfloor\s_1-K\rfloor, b=\lceil\s_2+K\rceil.
\]

Write
\[
B_1=\{n\in[b,a]\cap\Z: 1.5\cdot H(\nu;2^n|2^{n+1})>  H(\nu;2^{n-1}|2^{n})\}.
\]
If $n_1\ge n_2$ are two consecutive elements of $B_1$, then
\[
 H(\nu;2^{n_2+j}|2^{n_2+j+1})\le (1.5)^{-j}  H(\nu;2^{n_2}|2^{n_2+1})
\]
for all $0\le j< n_1-n_2$.
Thus
\[
\sum_{j=n_2}^{n_1-1}  H(\nu;2^{j}|2^{j+1}) \le 3 H(\nu;2^{n_2}|2^{n_2+1}).
\]
A similar argument shows that
\begin{equation}\label{eq:contrib-3}
\sum_{j=b}^{\min B_1 -1} H(\nu;2^{j}|2^{j+1}) \le 3.
\end{equation}
Hence we have
\[
\b(\s_1-\s_2)<H(\nu;2^{\s_2}|2^{\s_1})\le 3 \sum_{n\in B_1}  H(\nu;2^{n}|2^{n+1}) +2K+5.
\]
The term $2K+5$ on the right is the contribution of \eqref{eq:contrib-3} combined with
\begin{align*}
H(\nu;2^{b-K}|2^b),H(\nu;2^{a}|2^{a+K})\le& K,\\
H(\nu;2^{\s_2}|2^{b-K}), H(\nu;2^{a+K}|2^{\s_1})\le& 1.
\end{align*}

We define $B_2\subset B_1$ by the following procedure.
First we select an $n\in B_1$ such that
$H(\nu;2^n|2^{n+1})$ is maximal and declare that $n\in B_2$.
Then in each step, we consider all $n\in B_1$ that are of distance at least $2K$ from
the elements of $B_2$ already selected.
We choose among these elements one such that $H(\nu;2^n|2^{n+1})$ is maximal
and declare it to be an element of $B_2$.
We continue this procedure until there is no $n\in B_1$ of distance at least $2K$ to the
already selected elements of $B_2$.

It is easy to see that the set $B_2$ obtained this way is $2K$ separated and satisfies
\begin{align*}
\b(\s_1-\s_2)-(2K+5)\le&3 \sum_{n\in B_1}  H(\nu;2^{n}|2^{n+1})\\
\le& 3\cdot4K \sum_{n\in B_2}  H(\nu;2^{n}|2^{n+1}).
\end{align*}

We put
\[
B_3:=\{n\in B_2: H(\nu;2^{n}|2^{n+1})>\b/12\}.
\]
Since $|B_2\backslash B_3|\le|B_2|<(\s_1-\s_2)/2K$, we have
\begin{align*}
\sum_{n\in B_3}  H(\nu;2^{n}|2^{n+1})\ge& \frac{\b(\s_1-\s_2)}{12K}-\frac{\b}{12}\frac{\s_1-\s_2}{2K}-\frac{2K+5}{12K}\\
\ge& \frac{1}{24K}\b(\s_1-\s_2)-1.
\end{align*}

Finally, we define $B$ as the set of $n\in B_3$ such that
\[
H(\mu; t|2t) \le 1-\a 
\]
for all $t$ with $2^{n-1}\le t\le 2^{n+1}$.
We clearly have
\[
|B_3\backslash B|\le \cN_1\{ \s\in[\s_2,\s_1]:H(\mu;2^\s|2^{\s+1})>1-\a\}<c\b(\s_1-\s_2),
\]
where $c$ is the number that appears in Theorem \ref{th:low-entropy-convolution}.
If we choose this number sufficiently small depending only on $K$, which depends only on $\a$,
then $B$ satisfies \eqref{eq:B-large}.
\footnote{ \eqref{eq:B-large} holds if we choose both the constant $c$ in Theorem \ref{th:low-entropy-convolution}
and the constant $c'=c_{11}$ in the lemma to be less than $(48K)^{-1}$.
This is satisfied, because $K<2\log\a^{-1}+11\le13\log\a^{-1}$ and $c_{11}=c=(1000\log \a^{-1})^{-1}$.}
\end{proof}

\begin{proof}[Proof of Theorem \ref{th:low-entropy-convolution}]
We can approximate $\nu$ by a finitely supported measure so that we change
its entropy at scales larger than $2^{\sigma_2}$ only by an arbitrarily small amount.
We can use for example Lemma \ref{lm:perturb} with $r_1$ very small.
Therefore, we may assume that $\nu$ is finitely supported.

Let $B$ be as in Lemma \ref{lm:low-entropy-proof} and fix some $n\in B$. 
Since
\[
1.5\cdot H(\nu; 2^n|2^{n+1})\ge H(\nu;2^{n-1}|2^n),
\]
we can apply Proposition \ref{pr:bernoulli-decompose} with $r=2^n$ and
write
\[
\nu=\nu_0+\eta_1+\ldots+\eta_N,
\]
where each $\eta_i$ is a Bernoulli measure supported at a pair of points of distance between
$2^{n-1}$ and $2^{n+1}$, $\nu_0$ is a non-negative measure
and
\[
\|\eta_1\|+\ldots+\|\eta_N\|\ge c(\log \b^{-1})^{-1}H(\nu;2^n|2^{n+1}),
\footnote{$c_{12}=1/1000$. Here we need $c_{12}(\log \b^{-1})^{-1}\le1/128(\log(12/\b)+1)$, which holds since $\b\le1/2$.}
\]
where $c$ is an absolute constant.

By Proposition \ref{pr:conv-by-Bernoulli} (the conditions of the proposition are met by our choice of the value of $K$), we have
\[
H(\mu*\eta_i;2^{n-K}|2^{n+K})\ge \|\eta_i\|\cdot (H(\mu;2^{n-K}|2^{n+K})+\a/3)
\]
for each $i=1,\ldots,N$.
We combine this with the trivial estimate
(coming from Lemma \ref{lm:entropy-conv-nondecrease})
\[
H(\mu*\nu_0;2^{n-K}|2^{n+K})\ge\|\nu_0\|\cdot H(\mu;2^{n-K}|2^{n+K})
\]
and use superadditivity of entropy between scales of integral ratio
(Lemma \ref{lm:conditional-super-additive}) to obtain
\begin{equation}\label{eq:n-in-B}
H(\mu*\nu;2^{n-K}|2^{n+K})\ge H(\mu;2^{n-K}|2^{n+K}) + c(\log \b^{-1})^{-1}H(\nu;2^n|2^{n+1}),
\quad\footnote{$c_{13}=\a/3000$.}
\end{equation}
for some number $c$ that depends only on $\a$.

We note that $H(\mu*\nu;2^m|2^{m+1})\ge H(\mu;2^m|2^{m+1})$ for all $m\in\Z$
by Lemma \ref{lm:entropy-conv-nondecrease}.
We sum this for all $m\in[\s_2,\s_1]\cap\Z$ that is not in $[n-K,n+K)$ for any $n\in B$
together with \eqref{eq:n-in-B} for $n\in B$ and obtain also using \eqref{eq:B-large}
\begin{align*}
H(\mu*\nu;2^{\lceil\s_2\rceil}|&2^{\lfloor\s_1\rfloor})\\
\ge& H(\mu;2^{\lceil\s_2\rceil}|2^{\lfloor\s_1\rfloor})
+ c(\log \b^{-1})^{-1}\sum_{n\in B}H(\nu;2^n|2^{n+1})\quad\footnotemark\\
>& H(\mu;2^{\lceil\s_2\rceil}|2^{\lfloor\s_1\rfloor})
+ c\b(\log \b^{-1})^{-1}(\s_1-\s_2)-1.\quad\footnotemark
\end{align*}
\addtocounter{footnote}{-1}
\footnotetext{$c_{13}=\a/3000$.}
\stepcounter{footnote}
\footnotetext{$c_{14}=\a/(10^7\log\a^{-1})$. Here we need $c_{14}<c_{13}c_{11}=(\a/3000)/(1000\log\a^{-1})$.
Here $c_{11}$ is from Lemma \ref{lm:low-entropy-proof} on page \pageref{eq:B-large}.}
The claim of the theorem follows from this and the inequalities
\begin{align*}
H(\mu*\nu;2^{\s_2}|2^{\s_1})\ge& H(\mu*\nu;2^{\lceil\s_2\rceil}|2^{\lfloor\s_1\rfloor}),\\
H(\mu;2^{\s_2}|2^{\s_1})\le& H(\mu;2^{\lceil\s_2\rceil}|2^{\lfloor\s_1\rfloor})+2.
\end{align*}
\end{proof}

\section{Absolute continuity of Bernoulli convolutions}\label{sc:proof}

The purpose of this section is the proof of Theorem \ref{th:main}.
We begin by recalling the following result of Garsia, which links
absolute continuity of measures to entropy estimates.

\begin{prp}\label{pr:Garsia}
Let $\mu$ be a compactly supported probability measure on $\R$.
Then $\mu$ is absolutely continuous with density in the class $L\log L$
if and only if $\log r^{-1} - H(\mu;r)$ is bounded as $r\to 0$.
\end{prp}

This is just a small variation on \cite{garsia-entropy}*{Theorem I.5}, but we give
a proof for the reader's convenience.

\begin{proof}
Suppose that $\log r^{-1} - H(\mu;r)$ is bounded as $r\to 0$
and suppose to the contrary that there is a compact set $E\subset \R$
of Lebesgue measure $0$ such that $\mu(E)>0$.

Note that (see Lemma \ref{lm:second-def})
\[
\log r^{-1} - H(\mu;r)=\int_{-\infty}^\infty \mu*\chi_r(x)\log(\mu*\chi_r(x))dx,
\]
where $\chi_r$ is the density of the uniform distribution on the interval $[0,r]$.

We also observe that
\[
A_r:=\int_{E+[0,r]}\mu*\chi_r(x)dx=\frac{1}{r}\int_0^r\mu(E+[0,r]-t)dt \ge \mu(E).
\]
By Jensen's inequality applied to the function $x\mapsto x\log x$, we get
\begin{align*}
\int_{E+[0,r]} \mu&*\chi_r(x)\log(\mu*\chi_r(x))dx\\
=&m(E+[0,r])\cdot\frac{1}{m(E+[0,r])}\int_{E+[0,r]} \mu*\chi_r(x)\log(\mu*\chi_r(x))dx\\
\ge& m(E+[0,r])\cdot\frac{A_r}{m(E+[0,r])}\log \frac{A_r}{m(E+[0,r])},
\end{align*}
where $m(\cdot)$ denotes Lebesgue measure.

Let $[-B,B]$ be an interval containing the support of $\mu$, and denote
$\a=\min_{x\ge 0} x\log x$.
Then
\[
\int_{\R\backslash (E+[0,r])} \mu*\chi_r(x)\log(\mu*\chi_r(x))dx
\ge (2B+r) \a
\]

Combining our estimates we write
\[
\log r^{-1} - H(\mu;r)\ge \mu(E)\log\frac{\mu(E)}{\log m(E+[0,r])}+(2B+r) \a.
\]
This is unbounded, since $m(E+[0,r])\to 0$ as $r\to 0$, which is  a contradiction.

We have established that the boundedness of $\log r^{-1} - H(\mu;r)$ implies that
$\mu$ is absolutely continuous.
We know show that the density of $\mu$ belongs to $L\log L$.
By the Lebesgue differentiation theorem, then  $\mu*\chi_r\to \mu'$ almost everywhere,
where $\mu'$ is the density of $\mu$.
Using Fatou's lemma, we can write
\begin{align*}
\int \mu'(x)\log\mu'(x)dx\le&\liminf_{r\to 0}\int \mu*\chi_r(x)\log(\mu*\chi_r(x))dx\\
=&\liminf_{r\to 0}(\log r^{-1}-H(\mu;r)),
\end{align*}
which is bounded by assumption.
(Fatou's lemma applies, because $\mu'$ is compactly supported and $x\mapsto x\log x$ is bounded from
below.)
Since $\mu'$ is compactly supported, this implies that it is in the class $L\log L$.

For the converse, we note that if $\mu$ is absolutely continuous with class $L\log L$ density,
then we have
$H(\mu)\le H(\mu*\chi_r)$
by \eqref{lowerbound}.
This can be rewritten as 
\[
\int \mu'(x)\log\mu'(x)\ge\int \mu*\chi_r(x)\log(\mu*\chi_r(x))dx=\log r^{-1} - H(\mu;r),
\]
which proves the claim.
The last equation holds by Lemma \ref{lm:second-def}.
\end{proof}

In the rest of the section, we aim to verify the condition in this proposition for the
Bernoulli convolutions $\mu_{\l,p}$ with parameters that satisfy the hypothesis of Theorem \ref{th:main}.
To this end, we will show that
\begin{equation}\label{eq:polydecay}
H(\mu_{\l,p};r|2r)\ge1-(\log r^{-1})^{-2}
\end{equation}
for all $r$ small enough under the hypothesis of the theorem.
Summing this for $r=2^{-n}$, we clearly satisfy the condition in Proposition \ref{pr:Garsia}
proving Theorem \ref{th:main}.

In Section \ref{sc:HER}, we introduce the condition $k$-HE for probability measures on $\R$.
This condition is designed in a way to ensure that the convolution of two $k$-HE measures
satisfies $(k+1)$-HE.
The proof of this will be a direct application of Theorem \ref{th:high-entropy-convolution}.

We will also see that $k$-HE for sufficiently large $k$ depending on $r$ will imply
\eqref{eq:polydecay}.
At this point it will be left to show that we can decompose $\mu_{\l,p}$ as a convolution product
of sufficiently many measures each of which satisfies $0$-HE.

For $I\subset\R_{>0}$ we write $\mu^{I}$ for the law of the random variable
\[
\sum_{n\in\Z_{\ge 0}: \l^n\in I} \xi_n\l^n,
\]
where $\xi_0,\xi_1,\ldots$ is a sequence of independent random variables with
$\P(\xi_n=1)=p$ and $\P(\xi_n=-1)=1-p$.
(A similar notation was introduced in the unbiased case before.
From this point on, $\l$ and $p$ are considered fixed, so we suppress them in this
notation.)

If $I_1,\ldots, I_K\subset \R_{>0}$ are disjoint intervals, then there is a probability measure $\nu$
such that
\begin{equation}\label{eq:conv-decompose}
\mu_{\l,p}=\mu^{I_1}*\ldots *\mu^{I_K}*\nu.
\end{equation}
In Section \ref{sc:LER}, we will show that we can find intervals $I_j$ such that $\mu^{I_j}$
satisfies $0$-HE.
To that end, we will further decompose $\mu^{I_j}$ as a convolution product and use Theorem \ref{th:low-entropy-convolution},
the separation between the points in the support of $\mu^{I}$ and the estimates
of \cite{BV-entropy}*{Theorem 5} for the entropy $h_{\l,p}$ of the discrete random walk.

\subsection{The high entropy regime}\label{sc:HER}

We fix a large number $A$,\footnote{We can take $A=47$.}
whose value will be chosen depending only on the constant $C$
in Theorem \ref{th:high-entropy-convolution}.
We say that a probability measure $\mu$ supported on a compact subset of $\R$
satisfies the $k$-th high entropy inequality (or $k$-HE) at scale $r$ if
\[
H(\mu;t|2t)\ge 1-2^{-(2^k+3k+A)}
\]
for all $t$ with
\[
|\log t- \log r|\le A(2+\log\log\log r^{-1}-k)\log\log r^{-1}.
\]

Here and everywhere in what follows, we assume that (say) $r<2^{-4}$, hence $\log\log\log r^{-1}$
is defined and is at least $1$.

\begin{prp}\label{pr:high-entropy}
Let $\mu$ and $\nu$ be two compactly supported probability measures on $\R$ and let $r>0$
be a real number and let $k$ be an integer such that
\[
0\le k\le 1+\log\log\log r^{-1}.
\]
Suppose that $\mu$ and $\nu$ both satisfy $k$-{\rm HE} at scale $r$.

If the parameter $A$ fixed above is sufficiently large and $r$ is sufficiently small depending only on
the constant $C$ in Theorem \ref{th:high-entropy-convolution},
then $\mu*\nu$ satisfies $(k+1)$-{\rm HE} at scale $r$.
\end{prp}

\begin{proof}
We apply Theorem \ref{th:high-entropy-convolution} with
\[
\a=2^{-(2^k+3k+A)}
\]
at all scales (whose $\log$ is) between
\[
\log r \pm A(1+\log\log\log r^{-1}-k)\log\log r^{-1}.
\]

First we check that
\begin{equation}\label{eq:HE-first-condition}
A\log\log r^{-1}\ge 3\log \a^{-1}
\end{equation}
holds, hence $\mu$ and $\nu$ satisfy the conditions of Theorem \ref{th:high-entropy-convolution}.
We write
\[
\log\a^{-1}=2^k+3k+A\le 2\log\log r^{-1}+3\log\log\log r^{-1}+3+A.
\]
We see that \eqref{eq:HE-first-condition} holds provided $A> 6$ and $r$ is sufficiently small
depending on $A$.

The estimate in Theorem \ref{th:high-entropy-convolution} implies that $\mu*\nu$ satisfies $(k+1)$-HE
at scale $r$ provided
\[
2^{-(2^{k+1}+3(k+1)+A)}\ge C(\log \a^{-1})^3\a^2
=C(2^k+2k+A)^3 2^{-(2^{k+1}+6k+2A)}.
\quad\footnote{$C_{15}=10^{8}$ is the constant from Theorem \ref{th:high-entropy-convolution}.}
\]
This is equivalent to
\[
2^{3k+A-3}\ge C(2^k+3k+A)^3\quad\footnote{$C_{15}=10^{8}$.}
\]
or
\[
2^{A/3-1}\ge C(1+(3k+A)2^{-k}).\quad\footnote{$C_{16}=500>C_{15}^{1/3}$.}
\]
Fixing $k$ and increasing $A$, the left hand side grows faster than the right, hence
the inequality holds for $k=0$ and $k=1$ if we choose $A$
\footnote{$A=47$ works here.}
sufficiently large depending only on $C$.
However, the right hand side is maximal for $k=0$ or $k=1$ for any choice of $A$,
hence the inequality holds for all $k$.
\end{proof}

{}From now on, we assume that the parameter $A$ that appears
in the definition of $k$-HE is sufficiently large
so that Proposition \ref{pr:high-entropy} holds.
This is the only requirement we impose on $A$.

We aim to show that for $r$ sufficiently small, $\mu_{\l,p}$ satisfies  $k$-HE at scale $r$ for
$k=\lfloor\log\log\log r^{-1}+1\rfloor+1$.
This implies \eqref{eq:polydecay}.
Indeed:
\[
H(\mu_{\l,p};r|2r)\ge1-2^{-2^k}\ge1-(\log r^{-1})^{-2}.
\]

We achieve this by decomposing $\mu_{\l,p}$ as the convolution product of $2^k$ measures
that satisfy $0$-HE and another arbitrary measure
as in \eqref{eq:conv-decompose}, and  then use  Proposition \ref{pr:high-entropy} iteratively.

In the next section, we prove Proposition \ref{pr:low-entropy-regime}, which implies that
under the hypothesis of Theorem \ref{th:main}, the decomposition \eqref{eq:conv-decompose}
exists with $K=\lceil4\log\log r^{-1}\rceil$ such that each $\mu^{I_j}$ satisfies $0$-HE at scale $r$.
In light of the above comments, this proves \eqref{eq:polydecay} and Theorem \ref{th:main}
in turn.

\subsection{The low entropy regime}\label{sc:LER}

The aim of this section is to prove the following result, which completes the proof of Theorem \ref{th:main}.
In this section, the values of the constants given in the footnotes are valid under the additional hypothesis
that $\l$ is not the root of a polynomial with coefficients $-1$, $0$ and $1$ and $1/4\le p\le 3/4$.

\begin{prp}\label{pr:low-entropy-regime}
There is a number $c>0$ depending only on $p$ such that the following holds.
Let $\l<1$ be an algebraic number and suppose that
\[
\l> 1- c\min(\log M_\l,(\log (M_\l+1))^{-1}(\log\log(M_\l+2))^{-3}).
\quad\footnote{$c_{24}=10^{-37}$ under the additional assumption
that $\l$ is not a root of a polynomial with coefficients $-1$, $0$ and $1$ and $1/4\le p\le3/4$.}
\]
Suppose that $r>0$ is sufficiently small.

Then there are at least $4\log\log r^{-1}$ pairwise disjoint intervals $I$ such that
$\mu^I$ satisfies $0$-{\rm HE} at scale $r$.
\end{prp}

\subsubsection{}
In this section, we prove a technical lemma, which shows that for proving
that $\mu^{I}$ satisfies $0$-HE, it is enough to show that $\mu^{I'}$ has large entropy
at a single scale for some $I'$ slightly smaller than $I$.

We will use this to show that if  Theorem \ref{th:low-entropy-convolution}
cannot be applied for $\mu^{I'}$,
because the required upper bound on its entropy fails,
then $\mu^{I}$ satisfies $0$-HE for a corresponding interval $I$.

\begin{lem}\label{lm:HE0}
There is a number $\a_0>0$
\footnote{$\a_0=2^{47}/10$.}
such that the following holds.
Let  $0<\l,r,s,t<1$ be real numbers.
Suppose that
\[
\log t^{-1}>(\log\log r^{-1})^{2}, \quad r<s< t^4,\quad  
\l>1-\a_0,
\]
\[
H(\mu^{(t^2,t)};s|2s)>1-\a_0.
\]

Then the measure $\mu^{(t^3r/s,r/s)}$ satisfies $0$-{\rm HE} at scale $r$, provided $r$ is sufficiently small
depending only on $A$ (the parameter appearing in the definition of $k$-{\rm HE}).
\end{lem}

Informally, the proof goes as follows.
We need to show that
\[
H(\mu^{(t^3r/s,r/s)};\wt t|2\wt t)\ge 1-2^{-(A+1)}
\]
for all scales $\wt t$ near $r$.
Assuming that $\wt t/s$ is a power of $\lambda$, by \eqref{eq:scale-scaling}, we have
\[
H(\mu^{(t^3r/s,r/s)};\wt t|2\wt t)=H(\mu^{(t^3r/\wt t,r/\wt t)};s|2s),
\]
since $\mu^{\l^k I}$ is obtained from $\mu^I$ by scaling by $\lambda^k$.
In this case, the claim follows by Lemma \ref{lm:entropy-conv-nondecrease}.
We proceed with the details.

\begin{proof}
Let $\wt t$ be a number with
\[
|\log \wt t-\log r|\le (\log\log r^{-1})^2/2.
\]
If $r$ is sufficiently small, this includes all scales we need to consider in the definition of $0$-HE.

Let $k$ be an integer such that
\begin{equation}\label{eq:k-def}
|\log (s\l^k)-\log \wt t|<\log \l^{-1}.
\end{equation}
We then have
\[
|\log(\l^k)-\log (r/s)|< (\log\log r^{-1})^2/2+\log \l^{-1}<(\log\log r^{-1})^2<\log t^{-1},
\]
provided $r$ is sufficiently small.

Thus $(t^2\l^k,t\l^k)\subset (t^3r/s,r/s)$ and consequently
\begin{align*}
H(\mu^{(t^3r/s,r/s)};\l^ks|2\l^ks)\ge&
H(\mu^{(t^2\l^k,t\l^k)};\l^ks|2\l^ks)\\
=&
H(\mu^{(t^2,t)};s|2s)\\
>&1-\a_0.
\end{align*}
Here we used Lemma \ref{lm:entropy-conv-nondecrease} and then
\eqref{eq:scale-scaling}. 

We combine this with Lemma \ref{lemma:Lipschitz} and
\eqref{eq:k-def}, then use $\l>1-\a_0$ to obtain
\begin{align*}
H(\mu^{(t^3r/s,r/s)};\wt t|2\wt t)>&1-\a_0-2\log\l^{-1}\\
\ge& 1-\a_0 -2\log(1-\a_0)^{-1}\\
\ge&1-2^{-(A+1)}\quad\footnotemark
\end{align*}
\footnotetext{Using $-\log(1-x)\le 2x$, which is valid for $0\le x\le 1/2$, we have
$1-\a_0 -2\log (1-\a_0)^{-1}=1-\a_0 +2\log(1- \a_0) \ge 1- 2^{-47}/10-4\cdot2^{-47}/10=1-2^{-48}$.}
provided $\a_0$ is sufficiently small.
\end{proof}

\subsubsection{}

Our aim in this section is to show that for any sufficiently small number $t>0$,
there are many scales $s$ such that $H(\mu^{(t^2,t)};s|2s)>1-\a$, where $\a>0$ is
an arbitrary (but a previously fixed) number.
We will then combine this with Lemma \ref{lm:HE0} to find intervals $I$ such that
$\mu^I$ satisfies $0$-HE at suitable scales.

The argument is simple, but requires some technical calculations, which obscure
the ideas.
For this reason, we first explain the strategy without the detailed calculations.
We fix some carefully chosen parameters $0<\tau<t<1$ and $\ell\in\Z_{>0}$.
We consider numbers $a$ such that $(a\l^\ell,a]\subset (t^2,t)$.
We use certain Diophantine considerations going back to Garsia \cite{Garsia-arithmetic}
to bound from below the separation between the points in the support of the measure
$\mu^{(a\l^\ell,a]}$.
We set our parameters to ensure that this separation is at least $\tau$.

This allows us to estimate $H(\mu^{(a\l^\ell,a]};\tau)$ in terms of the Shannon entropy
$H(\mu^{(a\l^\ell,a]})\ge h_{\l,p} \ell$.
Recall that
\[
h_{\l,p}=\lim_{\ell\to\infty}\frac{H(\mu^{(\l^\ell,1]})}{\ell},
\]
where the sequence in the limit is monotone non-increasing.
We then plug in the bound on $h_{\l,p}$ from \cite{BV-entropy}.
We choose another parameter $\tau_1>\tau$ so that $\log\tau_1$ is at most half
the lower bound that we gave for $H(\mu^{(a\l^\ell,a]};\tau)$.
This will yield the bound
\[
H(\mu^{(a\l^\ell,a]};\tau|\tau_1)\ge \frac{c}{\log(M_\l+1)}\log\tau^{-1}.
\]

Then we consider a sequence $a_i$ such that the intervals $(a_i\l^\ell,a_i]\subset (t^2,t)$
are disjoint and apply Theorem \ref{th:low-entropy-convolution} repeatedly
with $\b=c/\log(M_\l+1)$ for the measures
\[
\mu=\mu^{(a_1\l^\ell,a_1]}*\ldots*\mu^{(a_i\l^\ell,a_i]},\quad
\nu=\mu^{(a_{i+1}\l^\ell,a_{i+1}]}.
\]
We will see that if $\l$ is sufficiently close to $1$, then we can find sufficiently many such intervals
so that the combined contributions of the entropy increments given by
Theorem \ref{th:low-entropy-convolution}
would exceed $\log\tau^{-1}$,
which is impossible.
This means that in one of the steps, the hypothesis of Theorem \ref{th:low-entropy-convolution}
must fail for $\mu=\mu^{(a_1\l^\ell,a_1]}*\ldots*\mu^{(a_i\l^\ell,a_i]}$, that is,
\[
H(\mu^{(a_1\l^\ell,a_1]}*\ldots*\mu^{(a_i\l^\ell,a_i]};s|2s)>1-\a
\]
for many $s\in(\tau,\tau_1)$.
This is what we wanted to do.

We turn to the details.
We first recall the following result of Garsia we alluded to above.
\begin{lem}[\cite{Garsia-arithmetic}*{Lemma 1.51}]\label{lm:Garsia}
Let $\l$ be an algebraic number and denote by $d$ the number of its Galois conjugates
that lie on the unit circle.

Then there is a number $c=c_\l$ such that the following holds.
Let $\ell\in\Z_{>0}$ and $b_0,\ldots, b_\ell\in\{-1,0,1\}$.
Then
\[
\Big|\sum_{j=0}^\ell b_j \l^j\Big|>c_\l \ell^{-d}M_\l^{-\ell}.
\]
\end{lem}

\begin{lem}\label{lm:separation}
For every $0<p<1$, there is a number $c>0$ such that the following holds.
Let $0<\l<1$ be an algebraic number and
let $0<\tau<1$ be a number.
Let
\[
\ell=\Big\lfloor\frac{\log\tau^{-1}}{2\log M_\l}\Big\rfloor.
\]

If $\tau$ is sufficiently small depending only on $\l$, then there is number
$\tau_1$ such that $\tau_1/\tau\in\Z$, and
\begin{align*}
\log\tau_1^{-1}\ge&\frac{c\log \tau^{-1}}{\log (M_\l+1)},\quad\footnotemark\\
H(\mu^{(a\l^\ell,a]};\tau|\tau_1)\ge&\frac{c\log \tau^{-1}}{\log (M_\l+1)}
\end{align*}
\footnotetext{$c_{17}=1/20$ in both lines if $\l$ is not a root of a polynomial with coefficients $-1$, $0$ and $1$ and $1/4\le p\le 3/4$.
\label{page:c17}}
holds for all $\tau^{1/3}<a<1$.
\end{lem}

\begin{proof}
Any two points in the support of $\mu^{(a\l^\ell,a]}$ are of the form
\[
\l^k\sum_{j=0}^{\ell-1}\omega_j\l^j,\quad \l^k\sum_{j=0}^{\ell-1}\omega_j'\l^j,
\]
where $k$ is a positive integer such that $\l^k\in(a\l,a]$ and $\omega_j,\omega_j'=\pm1$.
Hence the difference of these two points is
\[
\Big|\l^k\sum_{j=0}^{\ell-1}(\omega_j\l^j-\omega_j'\l^j)\Big|=2\l^k\Big|\sum_{j=0}^{\ell-1} b_j\l^j\Big|
>a\Big|\sum_{j=0}^{\ell-1} b_j\l^j\Big|,
\]
where $b_j\in\{-1,0,1\}$.

By Lemma \ref{lm:Garsia}, 
any two distinct points in the support of $\mu^{(a\l^\ell,a]}$
are of distance at least 
$c_\l \ell^{-d} M_\l^{-\ell}a$.
By $a\ge \tau^{1/3}$ and the choice of $\ell$,
this is greater than $\tau$,
provided $\tau$ is small enough (depending only on $\l$).

Thus
\[
H(\mu^{(a\l^\ell,a]};\tau)=H(\mu^{(a\l^\ell,a]})
\ge h_{\l,p} \ell.
\]
We plug in the bound for $h_{\l,p}$ from
\cite{BV-entropy}*{Theorem 5} (see also Remark 6 there for the biased case),
and obtain
\[
H(\mu^{(a\l^\ell,a]};\tau)
\ge c_p\min(1,\log M_\l) \ell\ge \frac{c_p\log\tau^{-1}}{3\log (M_\l+1)}
\quad\footnote{If $\l$ is not a  root of a polynomial with coefficients $-1$, $0$ and $1$,
then we do not need to apply \cite{BV-entropy}*{Theorem 5}.
In this case, all the points $\sum_{j:\l^j\in(a\l^\ell,a)}\pm\l^j$ are distinct,
hence $H(\mu^{(a\l^\ell,a]})=(-p\log p-(1-p)\log(1-p))|\{j:\l^j\in(a\l^\ell,a]\}|\ge \ell/2$,
provided $1/4\le p \le 3/4$.
Therefore, we can take $c_p=1/2$.}
\]
using again  the definition of $\ell$, provided $\tau$ is small enough (depending only on $\l$),
where $c_p$ is a constant
depending only on $p$.

Fix a number $\tau_1>0$.
We note that
\[
H(\mu^{(a\l^\ell,a]};\tau_1)<\log\tau_1^{-1}+C_\l
\]
for some number $C_\l$ depending only on $\l$.
Indeed, if $L$ is a number larger than the length of the interval $\supp\mu_{\l,p}$, then
$H(\mu^{(a\l^\ell,a]};L)\le 1$.
If we choose $L$ so that $L\tau_1^{-1}$ is an integer, then we can take $C_\l=1+\log L$.

Thus
\[
H(\mu^{(a\l^\ell,a]};\tau|\tau_1)\ge \frac{c_p \log \tau^{-1}}{3\log (M_\l+1)}-\log \tau_1^{-1}-C_\l.
\]
It is easily seen from this formula that a suitable choice of $\tau_1$ is possible.
\footnote{Indeed, we can set $\tau_1=B\tau$, where $B$ is the largest integer such that
$\log (B\tau)\le(\log \tau)/(20\log (M_\l+1))$ holds.}
\end{proof}

In the next lemma we apply Theorem \ref{th:low-entropy-convolution}
repeatedly for the measures $\mu^{(a\l^\ell,a]}$ and make use
of the entropy bounds provided by Lemma \ref{lm:separation}.

We introduce the shorthand
\[
K_\l=\log(M_\l+1)\log\log(M_\l+2),
\]
which we continue to use until the end of the paper.

\begin{lem}\label{lm:low-ent1}
For any numbers $0<\a<1/2$, $0<p<1$, there is a number $c>0$ such that the following holds.
Let $0<\l<1$ be an algebraic number.
Suppose that $0<\tau$ is sufficiently small depending only on $\lambda$, $\alpha$ and $p$.
Choose a number $0<t<1$ such that
\begin{equation}\label{eq:lambda-LE1}
\l> 1- c\frac{\log (t)\min(\log M_\l,1)}{\log (\tau)\log\log (M_\l+2)},
\quad\footnote{$c_{19}=\a/(10^{10}\log\a^{-1})$.}
\end{equation}
\[
t\ge\tau^{1/6}.
\]

Then there is an integer
\[
K\ge cK_\l^{-1}\log \tau^{-1}\quad\footnote{$c_{21}=1/(10^5\log \a^{-1})$.}
\]
and real numbers
\[
\tau^{c(\log (M_\l+1))^{-1}}>s_1>\ldots>s_K>\tau\quad\footnote{$c_{17}=1/20$.}
\]
such that $s_i>2s_{i+1}$ and
\[
H(\mu^{(t^2,t)};s_i|2s_i)>1-\a
\]
for all $i$.
\end{lem}

It may look confusing that the parameter $\tau$ is required to be small
depending on $\lambda$ in an uncontrolled fashion
and then $\lambda$ is bounded below by a quantity depending on $\tau$
in \eqref{eq:lambda-LE1}.
Observe, however, that the lower bound in \eqref{eq:lambda-LE1} depends only on the
ratio $\log t/\log \tau$.
Hence, we can first choose $\tau$, depending on $\lambda$ and then $t$ to make sure that the
hypotheses of the lemma hold,
which we will see is possible if $\lambda$ satisfies the assumptions in Proposition
\ref{pr:low-entropy-regime}.
Note also that \eqref{eq:lambda-LE1} may be considered a condition of the form $t\le\tau^a$
for some $a$ depending on $\lambda$.
This complements the bound $t\ge \tau^{1/6}$ in the next line.

\begin{proof}
Let
\[
\ell=\Big\lfloor\frac{\log\tau^{-1}}{2\log M_\l}\Big\rfloor
\]
as in Lemma \ref{lm:separation}.
We set
\[
N=\Big\lfloor\frac{\log t}{\ell\log \l}\Big\rfloor.
\]
We note that $N$ can be made arbitrarily large if the constant $c$ in the lemma is
sufficiently small, see \eqref{eq:N-bound} below.
For each
$i=1,\ldots, N$
we put
\[
I_i=(\l^{i\ell}t,\l^{(i-1)\ell}t].
\]
Note that $I_i=(a_i\l^\ell,a_i]$ for some number $a_i$ such that
\[
1\ge a_i\ge \l^{N\ell} t\ge t^2\ge \tau^{1/3}.
\]
Thus we can apply Lemma  \ref{lm:separation} and write
\[
H(\mu^{I_i};\tau|\tau_1)\ge c_0\log \tau^{-1}/\log (M_\l+1)
\footnote{$c_0=c_{17}=1/20$.}
\]
for all $i$, where $\tau_1$ is as in Lemma  \ref{lm:separation}
and $c_0$ is the constant $c$ from that lemma.
Moreover, $I_i\subset (t^2,t)$ for all $i$.

Since $H(\mu;\tau|\tau_1)\le\log(\tau_1/\tau)$ for any probability measure $\mu$, there is
$i\in\{0,\ldots, N-1\}$ such that
\begin{equation}\label{eq:limited-growth}
H(\mu^{I_1}*\ldots* \mu^{I_{i+1}};\tau|\tau_1)
\le H(\mu^{I_1}*\ldots*\mu^{I_{i}};\tau|\tau_1)+\frac{1}{N}\log (\tau_1/\tau).
\end{equation}
We combine the definitions of $N$ and $\ell$:
\begin{equation}\label{eq:1perN}
\frac{1}{N}\le\bigg\lfloor\frac{2\log (t)\log (M_\l)}{\log (\tau^{-1})\log (\l)}\bigg\rfloor^{-1}\le -\frac{\log (\tau)\log (\l)}{\log (t)\log (M_\l)}.
\quad\footnote{Provided $\log t/\log\tau$ is sufficiently large so that $N\ge 1$.
This definitely holds if $c_{19}=\a/(10^{10}\log\a^{-1})$ in \eqref{eq:lambda-LE1}.}
\end{equation}

Using $-\log(\l)\le2(1-\l)$, which is valid for $1/2\le \l\le 1$, \eqref{eq:lambda-LE1} implies
\[
-\frac{\log (\tau)\log(\l)}{\log (t)\min(\log M_\l,1)}<\frac{2c}{\log\log(M_\l+2)}.
\footnote{$c_{19}=\a/(10^{10}\log\a^{-1})$.}
\]
We combine this with
\[
\frac{\min(\log M_\l,1)}{\log M_\l}\le\frac{2}{\log(M_\l+1)}
\]
and use \eqref{eq:1perN} to get
\begin{equation}\label{eq:N-bound}
\frac{1}{N}\le \frac{4c}{K_\l}.\footnote{$c_{19}=\a/(10^{10}\log\a^{-1})$.}
\end{equation}

We apply Theorem \ref{th:low-entropy-convolution} for the measures
$\mu=\mu^{I_1}*\ldots* \mu^{I_{i}}$ and $\nu=\mu^{I_{i+1}}$
with $\b= c_0/\log (M_\l+1)$, $\s_1=\log \tau_1$ and $\s_2=\log \tau$.
We have already seen that the hypothesis of the theorem holds for $\nu$.
If the hypothesis on $\mu$ is also satisfied, then
\[
H(\mu^{I_1}*\ldots* \mu^{I_{i+1}};\tau|\tau_1)
\ge H(\mu^{I_1}*\ldots*\mu^{I_{i}};\tau|\tau_1)
+\frac{c_2\log(\tau_1/\tau)}{K_\l}-3.
\quad\footnote{$c_2=c_{18}=\a/(1.6\cdot10^9\log\a^{-1}).$
Indeed, we have $\log\b^{-1}=\log(\log(M_\l+1)20)\le \log\log(M_\l+2)+\log 20\le 8\log\log(M_\l+2)$,
because $\log(20)/\log\log(3)<7$, so we need here
$c_{18}\le c_{17}c/8$, where $c=\a/(10^7\log\a^{-1})$ is the constant in Theorem \ref{th:low-entropy-convolution}.}
\]
for some $c_2>0$ that depends on $c_0$ and the constant $c$ from
Theorem \ref{th:low-entropy-convolution}.
Note that
\[
\log(\tau_1/\tau)>c_0 \log(\tau^{-1})/\log(M_\l+1)-1,
\footnote{$c_0=c_{17}=1/20$. See page \pageref{page:c17}.}
\]
which follows from the conclusion of Lemma  \ref{lm:separation}.
Hence the term $-3$ in the above estimate becomes negligible, provided $\tau$ is small enough. 
We reached a contradiction with \eqref{eq:N-bound} and \eqref{eq:limited-growth} if $c$ is sufficiently small in \eqref{eq:N-bound}.
\footnote{We have contradiction if $4c_{19}<c_2=c_{18}$, which holds with the choice
$c_{19}=\a/(10^{10}\log\a^{-1})$ that we made.}

Hence the hypothesis on $\mu$ in Theorem \ref{th:low-entropy-convolution} fails,
and we can find a $1$-separated set $\s_1\ge\ldots \ge\s_K$ in
\[
\{ \s\in[\log \tau,\log \tau_1]:H(\mu^{I_1}*\ldots* \mu^{I_{i}};2^\s|2^{\s+1})>1-\a\}.
\]
of cardinality
\[
K>cK_\l^{-1}(\log\tau_1-\log\tau)
\quad\footnote{$c_{20}=1/(2\cdot 10^4\log\a^{-1})$. This requires $c_{20}\le c_{17}c$, where $c$ is the constant from
Theorem \ref{th:low-entropy-convolution}.
In the lemma, we need to set $c_{21}$ to satisfy $c_{21}\le c_{20}/2$, because $\log\tau^{-1}\le 2 (\log\tau_1-\log\tau )$.}
\]
with a constant $c$ that depends on $c_0$ and the constant in Theorem \ref{th:low-entropy-convolution}.
Since $I_j\subset(t^2,t)$ for all $j$, we have
\[
H(\mu^{(t^2,t)};2^\s|2^{\s+1})\ge H(\mu^{I_1}*\ldots* \mu^{I_{i}};2^\s|2^{\s+1})
\]
by Lemma \ref{lm:entropy-conv-nondecrease} for all $\s$.
This concludes the proof.
\end{proof}

We combine Lemmata \ref{lm:HE0} and \ref{lm:low-ent1} in the next lemma.

\begin{lem}\label{lm:low-ent2}
Let $\a_0$ \footnote{$\a_0=2^{-47}/10$}
be the number from Lemma \ref{lm:HE0} and let $c_0$ be a number such that
Lemma \ref{lm:low-ent1} holds with $\a=\a_0$ and $c=c_0$.
Let $\l<1$ be an algebraic number and $0<r,t,\tau<1$ be real numbers.
Suppose that $r$ is sufficiently small, and
\begin{equation}\label{eq:rtau-bounds}
2^{-(\log r^{-1})^{1/2}}>\tau>r,
\end{equation}
\begin{equation}\label{eq:ttau-bounds}
(\log\tau^{-1})^{1/2}<\log t^{-1} < \frac{c_0}{4K_\l}\log \tau^{-1},
\quad\footnote{$c_0=c_{17}=1/20$.}
\end{equation}
\[
\l> 1- c_0\frac{\log (t)\min(1,\log M_\l)}{\log (\tau)\log\log( M_\l+2)}.
\quad\footnote{$c_0=c_{22}=10^{-27}$. This requires $c_{22}\le \a_0/(10^{10}\log\a_0^{-1})$.}
\]

Then there are at least
\[
\frac{c_0K_\l^{-1}\log \tau}{4\log t}
\quad\footnote{$c_0=c_{23}=10^{-7}$.}
\]
 pairwise disjoint intervals
\[
I\subset(r/\tau^{c_0(\log (M_\l+1))^{-1}},r/\tau)
\quad\footnote{$c_0=c_{17}=1/20$.}
\]
such that the measure
$\mu^{I}$ satisfies $0$-{\rm HE} at scale $r$.
\end{lem}

\begin{proof}
By \eqref{eq:ttau-bounds}, we have
\[
t>\tau^{c_0/4K_\l}\ge \tau^{1/6}
\]
so we can apply Lemma \ref{lm:low-ent1}, and obtain an integer
\[
K\ge c_0K_\l^{-1}\log \tau^{-1}
\quad\footnote{$c_0=c_{23}=10^{-7}$.
This requires $c_{23}\le 1/(10^5\log\a_0^{-1})$.}
\]
and real numbers
\[
\tau^{c_0(\log (M_\l+1))^{-1}}>s_1>\ldots>s_K>\tau
\quad\footnote{$c_0=c_{17}=1/20$.}
\]
such that $s_i>2s_{i+1}$ and
\[
H(\mu^{(t^2,t)};s_i|2s_i)>1-\a_0
\]
for all $i$.

We put $I_i=(t^3r/s_i,r/s_i)$.
We note that
\[
r/s_i\in(r/\tau^{c_0(\log (M_\l+1))^{-1}},r/\tau)
\]
for all $i=1,\ldots, K$, hence $I_i\subset (r/\tau^{c_0(\log (M_\l+1))^{-1}},r/\tau)$
for all $i\ge 3\log t^{-1}$.

We combine \eqref{eq:ttau-bounds} and \eqref{eq:rtau-bounds} and obtain
\[
t<2^{-(\log r^{-1})^{1/4}}<2^{-(\log\log r^{-1})^2},
\]
if $r$ is sufficiently small.
Moreover, \eqref{eq:ttau-bounds} implies that
\[
t^4>\tau^{c_0(\log(M_\l+1)^{-1})}>s_i
\]
for all $i$.
Thus Lemma \ref{lm:HE0} applies and
$\mu^{I_i}$ satisfy $0$-HE at scale $r$ for each $i$.

We observe that $I_i$ and $I_j$ are disjoint provided $2^{|i-j|}\ge t^{-3}$, 
that is $|i-j|\ge3\log t^{-1}$.
This shows that the intervals $I_{\lceil 3\log t\rceil j}$ satisfy all the requirements of the lemma.
\end{proof}

\subsubsection{Proof of Proposition \ref{pr:low-entropy-regime}}
With $c_0$ as in Lemma \ref{lm:low-ent2},
we put $a=c_0(\log( M_\l+1))^{-1}$. \footnote{$c_0=c_{17}=1/20$.}

We set $N$ in such a way that
\[
2^{-(\log r^{-1})^{1/2}}>r^{a^N},
\]
i.e. take $N=\lfloor\log\log r^{-1}/2\log a^{-1}\rfloor$.
We apply Lemma \ref{lm:low-ent2} for $n=0,1,\ldots, N$ with
$\tau=r^{a^n}$.
The choice of $a$ guarantees that the intervals $(r/\tau^{c_0(\log(M_\l+1))^{-1}},r/\tau)$
are pairwise disjoint for the different values of $n$.

For each $n$, we will set $t$ in Lemma \ref{lm:low-ent2} in such a way that
\begin{equation}\label{eq:intervals}
\frac{c_0K_\l^{-1}\log \tau}{4\log t}\cdot \frac{\log\log r^{-1}}{2\log a^{-1}}
\ge 4\log\log r^{-1}.\quad\footnote{$c_0=c_{23}=10^{-7}$.}
\end{equation}
If we can satisfy the conditions of Lemma \ref{lm:low-ent2} for each $n$,
then \eqref{eq:intervals} is a lower bound on the number of disjoint intervals $I$
such that $\mu^I$ satisfies $0$-HE at scale $r$, which proves the proposition.

Taking
\[
\frac{\log t}{\log \tau}=\frac{c_0}{32K_\l\log a^{-1}}
\quad\footnote{$c_0=c_{23}=10^{-7}$.}
\]
we satisfy \eqref{eq:intervals} and \eqref{eq:ttau-bounds}.
(The left hand side of  \eqref{eq:ttau-bounds} holds if $r$ and hence $t$ and $\tau$
are sufficiently small in terms of $\l$, $p$ and $a$.)
It remains to verify that the lower bound on $\l$
required in Lemma \ref{lm:low-ent2} holds, i.e. that
\begin{equation}\label{eq:lambda-requirement}
1-\l<c_0\frac{\log(t)\min(1,\log M_\l)}{\log (\tau)\log\log (M_\l+2)}.
\quad\footnote{$c_0=c_{22}=10^{-27}$.}
\end{equation}
Combining the definitions of $a$, $t$ and $\tau$, we write
\[
c_0\frac{\log(t)\min(1,\log M_\l)}{\log(\tau)\log\log (M_\l+2)}
>c\frac{\min(1,\log M_\l)}{\log(M_\l+1)(\log\log(M_\l+2))^3}
\quad\footnote{By the definition of $a$, $\log(a^{-1})<\log(20)+\log\log(M_\l+2)<8\log\log(M_\l+2)$.
Hence $\log t/\log \tau >c_{23}/(32\cdot 8 \log(M_\l+1)\log\log(M_\l+2)^2)>10^{-10}/\log(M_\l+1)\log\log(M_\l+2)^2)$.
The constant $c_0$ in this line is from \eqref{eq:lambda-requirement} and it is equal to $c_{22}=10^{-27}$.
Therefore, the claim holds with $c=10^{-37}$.
}
\]
for some number $c$ depending only on $c_0$.
Hence \eqref{eq:lambda-requirement} holds indeed
by the assumptions of the proposition.
\footnote{If $\l$ is not a root of a polynomial with coefficients $-1$, $0$ and $1$, then necessarily $M_\l\ge 2$,
hence $\min(1,\log M_\l)=1$, so in this special case we can set $c_{24}=10^{-37}$ in the proposition.}

\bibliography{bibfile}

\bigskip

\noindent{\sc Centre for Mathematical Sciences,
Wilberforce Road, Cambridge CB3 0WA,
UK}\\
{\em e-mail address:} pv270@dpmms.cam.ac.uk

\end{document}